\documentclass[12pt]{article}
\usepackage{amssymb}
\usepackage{amsmath}
\usepackage{amsthm}
\usepackage{color}
\usepackage{comment}
\usepackage{graphicx}
\begin{document}

\def\cL{\mathcal{L}}
\def\cM{\mathcal{M}}
\def\cN{\mathcal{N}}
\def\cO{\mathcal{O}}
\def\cObig{\mathcal{O} \hspace{-.5mm}}
\def\cS{\mathcal{S}}
\def\cX{\mathcal{X}}
\def\cY{\mathcal{Y}}
\def\ee{\varepsilon}

\newcommand{\removableFootnote}[1]{}

\newtheorem{theorem}{Theorem}

\title{
Unfolding homoclinic connections formed by corner intersections in piecewise-smooth maps.
}
\author{
D.J.W.~Simpson\\
Institute of Fundamental Sciences\\
Massey University\\
Palmerston North\\
New Zealand}
\maketitle




\begin{abstract}

The stable and unstable manifolds of an invariant set of a piecewise-smooth map are themselves piecewise-smooth. Consequently, as parameters of a piecewise-smooth map are varied, an invariant set can develop a homoclinic connection when its stable manifold intersects a non-differentiable point of its unstable manifold (or vice-versa). This is a codimension-one bifurcation analogous to a homoclinic tangency of a smooth map, referred to here as a homoclinic corner. This paper presents an unfolding of generic homoclinic corners for saddle fixed points of planar piecewise-smooth continuous maps. It is shown that a sequence of border-collision bifurcations limits to a homoclinic corner and that all nearby periodic solutions are unstable.

\end{abstract}



\section{Introduction}
\label{sec:intro}
\setcounter{equation}{0}

A transverse intersection between the stable and unstable manifolds of an invariant set
implies the existence of a Smale horseshoe and chaotic dynamics \cite{PaTa93,Ro04}.
For smooth maps, transverse intersections are born in homoclinic tangencies
where stable and unstable manifolds intersect tangentially.
For piecewise-smooth (PWS) maps, however, stable and unstable manifolds can be continuous but non-differentiable.
Transverse intersections may be created when
one manifold intersects a non-differentiable point of the other manifold.
Locally one manifold has a V-shape, instead of the U-shape of a quadratic tangency,
and here the bifurcation is referred to as a homoclinic corner.
The purpose of this paper is to unfold a generic homoclinic corner.

For simplicity only continuous planar maps are considered.
Specifically, this paper considers maps of the form
\begin{equation}
\begin{bmatrix} x \\ y \end{bmatrix} \mapsto f(x,y;\xi) =
\begin{cases}
f_1(x,y;\xi) \;, & (x,y) \in \cM_1 \\[-1.6mm]
\hspace{9mm} \vdots \\[-2.3mm]
f_m(x,y;\xi) \;, & (x,y) \in \cM_m
\end{cases} \;,
\label{eq:f}
\end{equation}
where the regions $\cM_1,\ldots,\cM_m$ partition the domain $\cM \subset \mathbb{R}^2$.
Each $f_j : \cM_j \times \mathbb{R} \to \cM$ is smooth
and smoothly extendable into an open set containing the closure of $\cM_j$\removableFootnote{
I.e.~there are no singularities on the switching manifolds.
This condition is needed so that we can construct a PW-Taylor series for $f^r$.
Note that $D f_j$ can be unbounded if $\cM_j$ is unbounded.
}.
The boundaries of the $\cM_j$ are one-dimensional PWS manifolds,
termed switching manifolds, and on the switching manifolds $f$ is continuous.
Maps of the form (\ref{eq:f})
provide useful mathematical models of discrete-time phenomena with switching events,
particularly in social sciences \cite{PuSu06},
and arise as return maps of PWS systems of differential equations \cite{DiBu08}.

First let us discuss border-collision bifurcations (BCBs) of (\ref{eq:f}).
As the parameter $\xi$ is varied, (\ref{eq:f}) undergoes a BCB when a fixed point collides with a switching manifold \cite{Si16}.
By expanding (\ref{eq:f}) about the BCB and truncating the expansion to leading order,
one obtains a piecewise-linear map that approximates the local dynamics.
Structurally stable dynamics of the piecewise-linear approximation
are exhibited by (\ref{eq:f}) near the BCB.
If the BCB is non-degenerate,
then the piecewise-linear approximation can be transformed to
\begin{equation}
\begin{bmatrix} x \\ y \end{bmatrix} \mapsto
\begin{cases}
\begin{bmatrix} \tau_L & 1 \\ -\delta_L & 0 \end{bmatrix}
\begin{bmatrix} x \\ y \end{bmatrix} +
\begin{bmatrix} \mu \\ 0 \end{bmatrix} \;, & x \le 0 \\
\begin{bmatrix} \tau_R & 1 \\ -\delta_R & 0 \end{bmatrix}
\begin{bmatrix} x \\ y \end{bmatrix} +
\begin{bmatrix} \mu \\ 0 \end{bmatrix} \;, & x \ge 0
\end{cases} \;,
\label{eq:bcNormalForm}
\end{equation}
where $\tau_L, \delta_L, \tau_R, \delta_R, \mu \in \mathbb{R}$ are parameters.
The map (\ref{eq:bcNormalForm}) is known as the two-dimensional border-collision normal form \cite{NuYo92,BaGr99}.

\begin{figure}[b!]
\begin{center}
\setlength{\unitlength}{1cm}
\begin{picture}(14,6)
\put(0,0){\includegraphics[height=6cm]{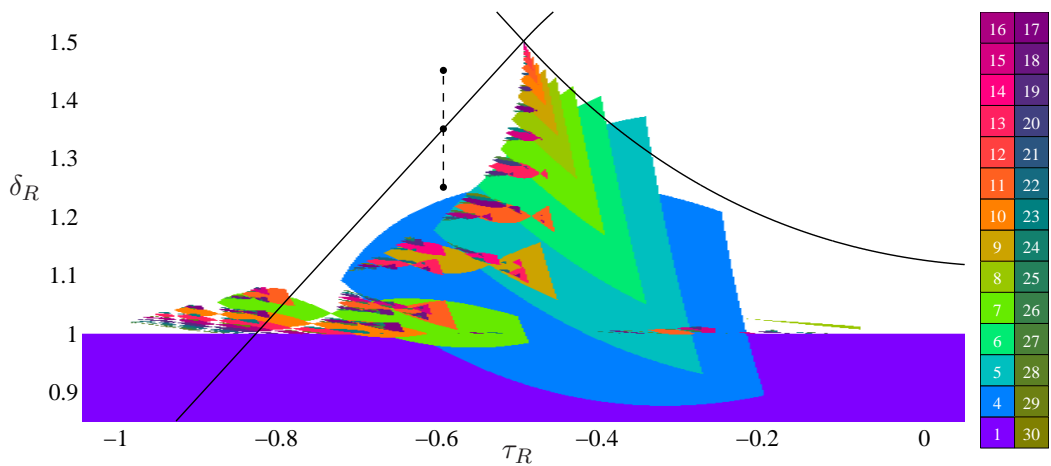}}
\put(6.55,0){\small $\tau_R$}
\put(0,3.5){\small $\delta_R$}
\end{picture}
\caption{
A bifurcation set of the two-dimensional border-collision normal form (\ref{eq:bcNormalForm})
with (\ref{eq:param}).
The black curves are loci of homoclinic corners.
The three black dots indicate the parameter values of Fig.~\ref{fig:pptRmp6}.
The coloured regions are mode-locking regions within which
(\ref{eq:bcNormalForm}) has a stable periodic solution
of a period indicated by the grid on the right.
These are shown up to period $30$.
At points where mode-locking regions overlap, the highest period is indicated.
Only mode-locking regions corresponding to ``rotational'' periodic solutions \cite{Si16,SiMe09}
are shown, because these can be computed efficiently and accurately
and other mode-locking regions that exist over the given parameter range
were found to be relatively small.
\label{fig:tonguesExact}
}
\end{center}
\end{figure}

The border-collision normal form arises in the analysis below.
But it is also an instance of the general form (\ref{eq:f}),
and so here it is used to illustrate homoclinic corners.
As an example, with
\begin{equation}
\tau_L = 2 \;, \qquad
\delta_L = 0.75 \;, \qquad
\mu = 1 \;,
\label{eq:param}
\end{equation}
and any $\tau_R, \delta_R \in \mathbb{R}$,
the map (\ref{eq:bcNormalForm}) has a fixed point $(x,y) = (-4,3)$
with eigenvalues $\lambda = 0.5$ and $\sigma = 1.5$.
The stable and unstable manifolds of the fixed point exhibit a homoclinic corner
along the two curves shown in Fig.~\ref{fig:tonguesExact}.
In order to indicate attractors, this figure also shows mode-locking regions -- where stable periodic solutions exist.
However, as explained below these have little relevance to the homoclinic corners\removableFootnote{
The sequence of mode-locking regions approaching $(\tau_R,\delta_R) = (-0.5,1.5)$
correspond to single-round periodic solutions.
The smaller regions that partially overlap the left parts of these correspond to multi-round periodic solutions.
}
(except near the codimension-two point $(\tau_R,\delta_R) = (-0.5,1.5)$
where the intersecting branches of the stable and unstable manifolds are coincident).

With $\tau_R = -0.6$, for example, a homoclinic corner occurs at $\delta_R = 1.35$.
Fig.~\ref{fig:pptRmp6} shows phase portraits at, and on either side of, the homoclinic corner.
On one side of the homoclinic corner the stable and unstable manifolds of the fixed point do not intersect.
A chaotic attractor\removableFootnote{
With $(\tau_R,\delta_R) = (-0.6,1.35)$,
I obtained a value of $0.212$ for the Lyapunov exponent of the attractor, {\sc goLyap.m}.

Immediately below most of the homoclinic corner curve
(where no stable periodic solutions are shown),
there appears to exist a chaotic attractor
(numerically I obtained positive Lyapunov exponents).
In a substantial part of parameter space this attractor coexists
with one or more stable periodic solutions.
The bifurcations responsible for the destruction of the attractor
with a decrease in the value of $\delta_R$ appear to be non-trivial and varied.
The chaotic attractor is destroyed in the homoclinic corner
(exactly at the homoclinic corner it exists).
Just above the homoclinic corner the stable manifold protrudes into
the region of phase space where the attractor was.
Here forward orbits can undergo transient chaotic motion
(for an extremely long time such that numerical verification of destruction at the homoclinic corner is difficult)
before becoming unbounded.
}
appears to be the only attractor of (\ref{eq:bcNormalForm}) for all $1.25 \le \delta_R \le 1.35$.
This attractor is destroyed at the homoclinic corner in a crisis \cite{GrOt83}.
On the other side of the homoclinic corner there is transient chaos and typical forward orbits diverge.

\begin{figure}[b!]
\begin{center}
\setlength{\unitlength}{1cm}
\begin{picture}(16,5.4)
\put(0,0){\includegraphics[height=5cm]{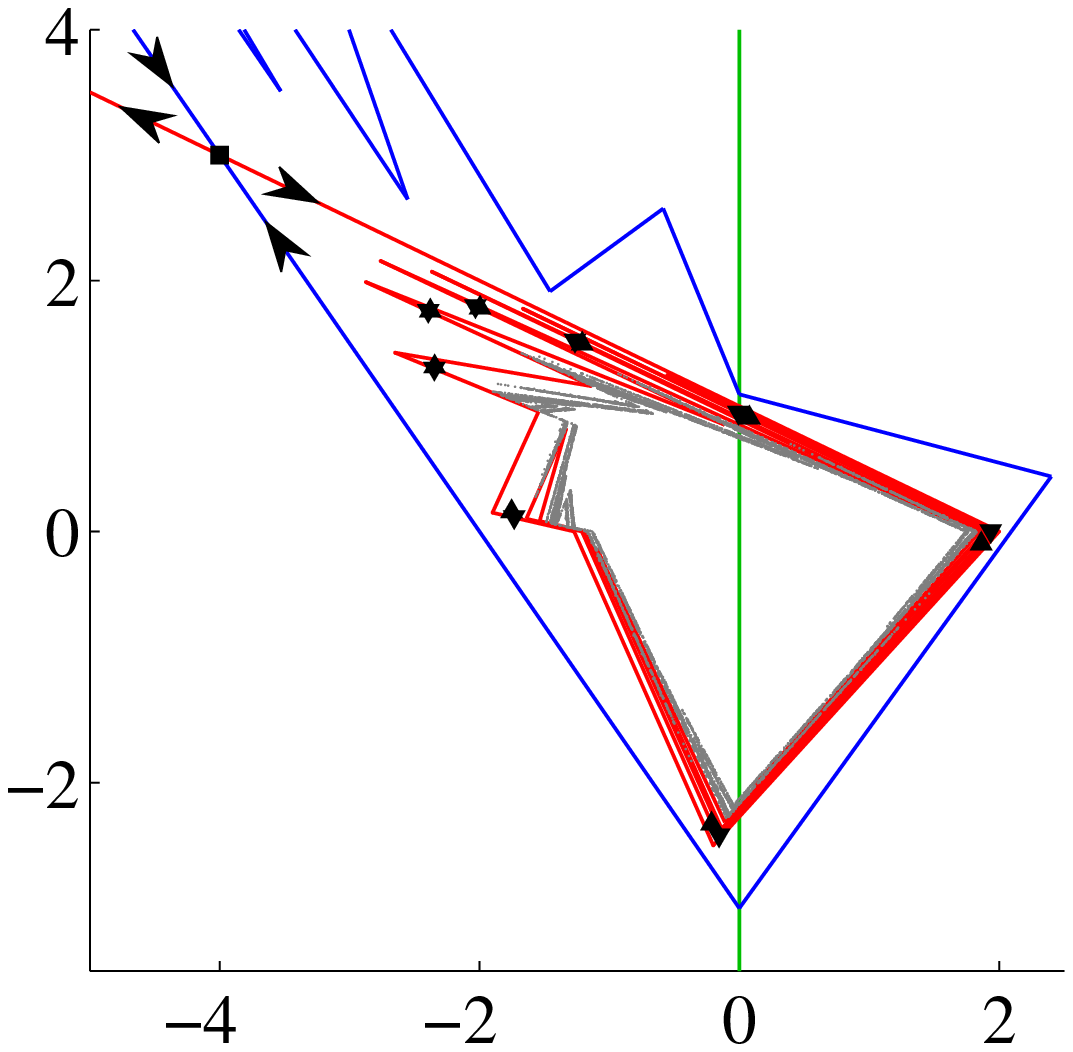}}
\put(5.5,0){\includegraphics[height=5cm]{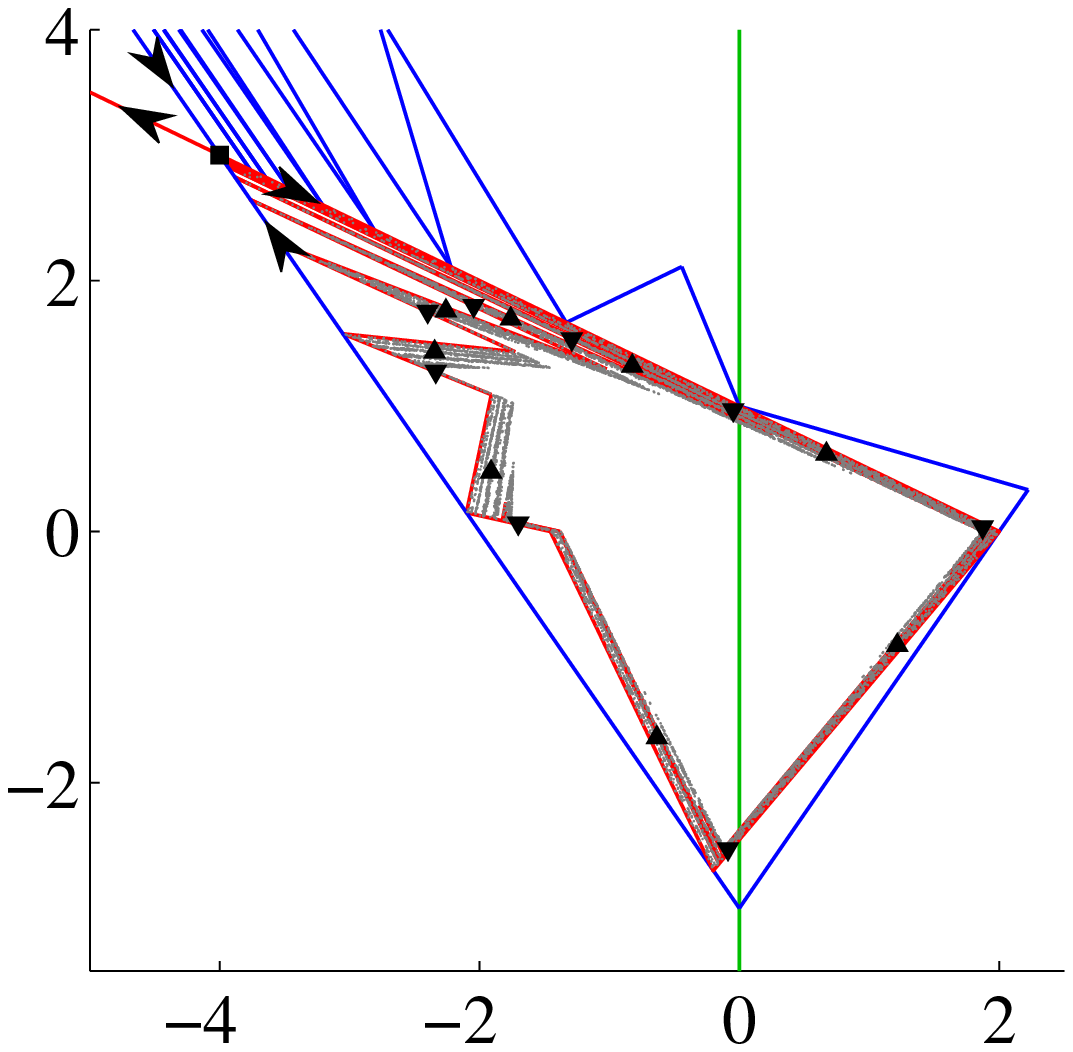}}
\put(11,0){\includegraphics[height=5cm]{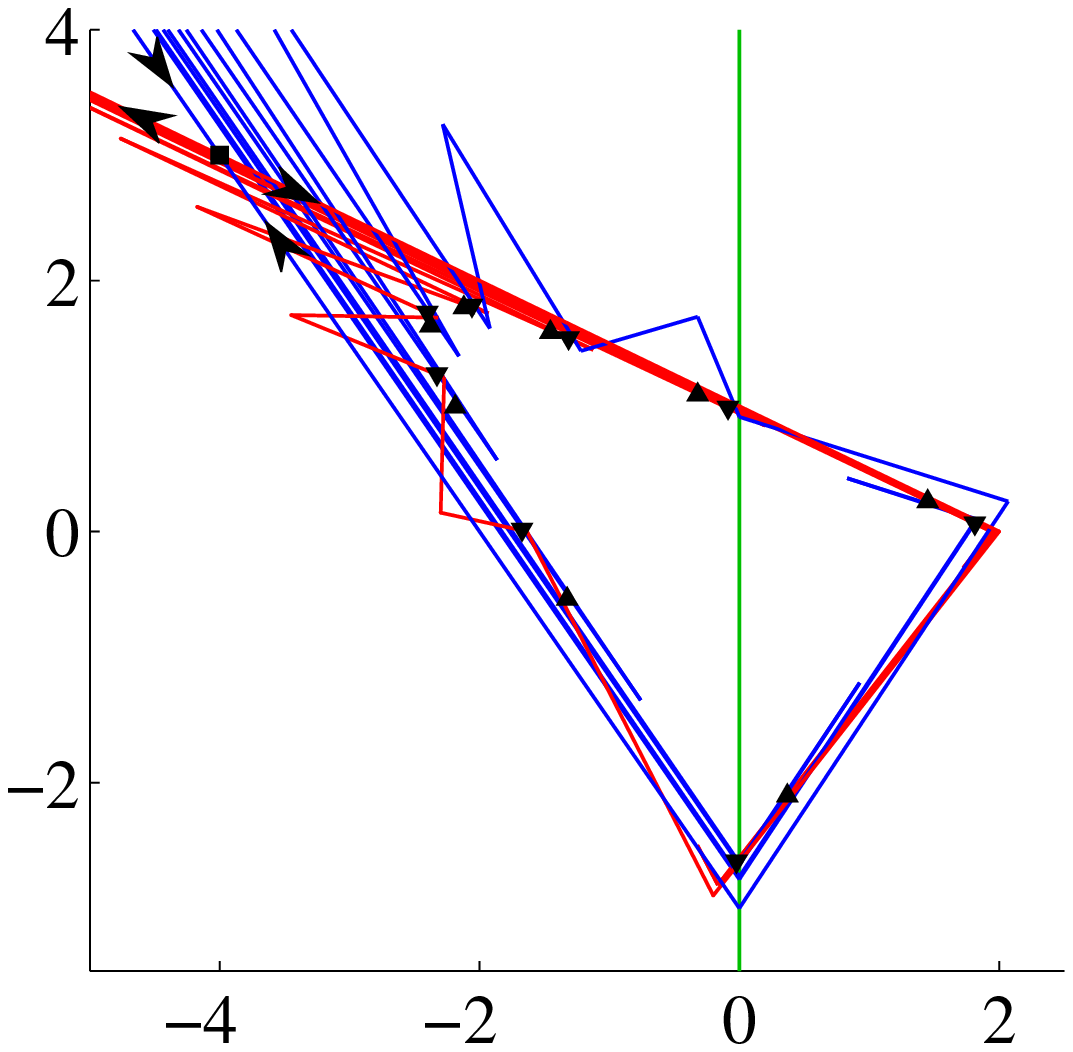}}
\put(2.75,0){\small $x$}
\put(0,3.06){\small $y$}
\put(8.25,0){\small $x$}
\put(5.5,3.06){\small $y$}
\put(13.75,0){\small $x$}
\put(11,3.06){\small $y$}
\put(2,5.2){\small $\delta_R = 1.25$}
\put(7.5,5.2){\small $\delta_R = 1.35$}
\put(13,5.2){\small $\delta_R = 1.45$}
\end{picture}
\caption{
Phase portraits of (\ref{eq:bcNormalForm}) with (\ref{eq:param}),
$\tau_R = -0.6$, and three different values of $\delta_R$.
In each phase portrait the stable and unstable manifolds of the saddle fixed point $(x,y) = (-4,3)$ are shown.
The left and middle phase portraits show
the	first $10^4$ points of the forward orbit of $(x,y) = (0,0)$,
with the first $10^3$ points removed. 
Each phase portrait also shows two single-round periodic solutions of period $n = 8$,
see Fig.~\ref{fig:bifDiagdRmp6}.
\label{fig:pptRmp6}
}
\end{center}
\end{figure}

The destruction of a chaotic attractor in a homoclinic corner is described in \cite{ZhMo08}
and other transitions are also possible.
For instance, as a homoclinic corner of Fig.~\ref{fig:tonguesExact} with $\delta_R < 1$ is crossed,
the basin of attraction of a stable fixed point appears to become fractal.
Other numerical explorations have revealed that
an invariant circle can be destroyed at a homoclinic corner \cite{ZhMo08,ZhMo06b}
and the number of bands in a chaotic attractor can change \cite{MaSu98}.

Homoclinic corners may be viewed as continuous but non-differentiable distortions of homoclinic tangencies.
For this reason it is expected that the two types of bifurcation exhibit similar unfoldings.
Consider for a moment a homoclinic tangency for a fixed point in a smooth two-dimensional map
with eigenvalues $\lambda$ and $\sigma$ that satisfy $0 < \lambda < 1 < \sigma$.
Suppose that the map depends smoothly on a scalar parameter $\xi$ and that the tangency occurs at $\xi = 0$.
The simplest bifurcations near $\xi = 0$ involve single-round periodic solutions.
These involve one excursion far from the fixed point.
Generically, for sufficiently large values of $n \in \mathbb{Z}^+$,
a pair of single-round periodic solutions of period $n$
is created in a saddle-node bifurcation at some $\xi = \xi_n$,
with $\xi_n \to 0$ as $n \to \infty$.
If $\lambda \sigma < 1$, then one of these periodic solutions is stable,
and there exists $C \ne 0$ such that 
\begin{equation}
\xi_n = C \sigma^{-n} + \cObig \left( \lambda^n \right) + \cObig \left( \sigma^{-2n} \right) \;,
\label{eq:scalingLaw}
\end{equation}
see \cite{GaSi72,GaSi73,GaWa87}.
Near each $\xi = \xi_n$, the stable period-$n$ solution loses stability in a period-doubling bifurcation,
beyond which there is a period-doubling cascade leading to a chaotic attractor
that is subsequently destroyed in a homoclinic tangency.
The dynamics are well approximated by a one-dimensional quadratic map
that can be used to explain this bifurcation sequence \cite{PaTa93}.

\begin{figure}[b!]
\begin{center}
\setlength{\unitlength}{1cm}
\begin{picture}(14,6)
\put(0,0){\includegraphics[height=6cm]{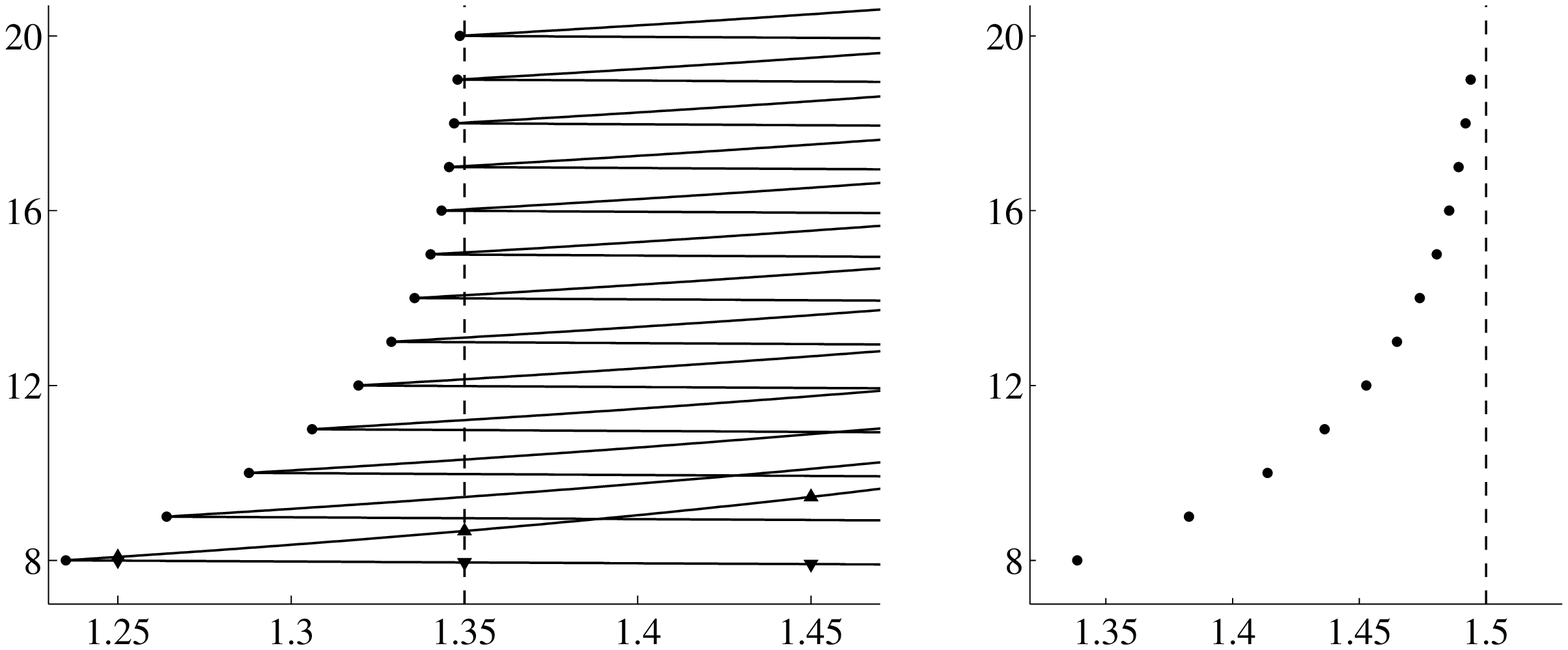}}
\put(1.5,5.9){\sf \bfseries A}
\put(9.8,5.9){\sf \bfseries B}
\put(4.4,.1){\small $\delta_R$}
\put(0,3.4){\small $n \hspace{-.3mm} + \hspace{-.3mm} x$}
\put(10.7,0){\small $\frac{\delta_{R_n} - \delta_{R,HC}}{\delta_{R_{n+1}} - \delta_{R,HC}}$}
\put(13.43,.59){\scriptsize $=$}
\put(13.7,.57){\footnotesize $\sigma$}
\put(8.7,3.4){\small $n$}
\end{picture}
\caption{
Panel A is a bifurcation diagram of (\ref{eq:bcNormalForm}) with (\ref{eq:param}) and $\tau_R = -0.6$.
For all $n \ge 8$, two single-round periodic solutions of period-$n$ are created in BCBs at values $\delta_{R_n}$
(indicated with black dots) that limit to $\delta_{R,HC} = 1.35$.
Each periodic solution involves one point on the switching manifold ($x = 0$) at its corresponding BCB.
The $x$-value of this point, plus the period $n$, is plotted as a function of $\delta_R$.
The triangles highlight the period-$8$ solutions plotted in Fig.~\ref{fig:pptRmp6}.
Panel B demonstrates the scaling law (\ref{eq:scalingLaw}).
\label{fig:bifDiagdRmp6}
}
\end{center}
\end{figure}

Fig.~\ref{fig:bifDiagdRmp6}-A is a bifurcation diagram showing
single-round periodic solutions near the homoclinic corner of Fig.~\ref{fig:pptRmp6}.
These are created in pairs at BCBs.
Fig.~\ref{fig:bifDiagdRmp6}-B shows that the fraction
$\frac{\delta_{R_n} - \delta_{R,HC}}{\delta_{R_{n+1}} - \delta_{R,HC}}$,
appears to converge to $\sigma$ as $n \to \infty$,
where $\delta_{R,HC} = 1.35$ and $\delta_{R_n}$ denotes the BCB value at which period-$n$ solutions are created.
This suggests that the BCBs adhere to the same scaling law as the saddle-node bifurcations (\ref{eq:scalingLaw}).

For this example $\lambda \sigma < 1$, yet all the periodic solutions
indicated in Fig.~\ref{fig:bifDiagdRmp6}-A are unstable.
The same observation was made in \cite{Si16}
for a homoclinic corner to a period-$3$ solution of (\ref{eq:bcNormalForm}).
In contrast, near the codimension-three homoclinic points of (\ref{eq:bcNormalForm}) described in \cite{Si14,Si14b},
single-round periodic solutions are stable and can coexist in arbitrarily large numbers.

The remainder of this paper is organised as follows.
The next section introduces equations useful for analysing
the dynamics near an arbitrary homoclinic corner of (\ref{eq:f}).
Section \ref{sec:assumptions} provides conditions on the parameters in these equations that ensure genericity.
In \S\ref{sec:unstable} single-round periodic solutions are described
and it is proved that all periodic solutions near a homoclinic corner (single-round and multi-round) are unstable.
In \S\ref{sec:bcbs} BCBs of single-round periodic solutions are analysed.
It is shown that the BCBs indeed satisfy the scaling law (\ref{eq:scalingLaw})
and that the periodic solutions are saddles with stable and unstable manifolds that intersect transversally.
Finally, \S\ref{sec:conc} provides a summary and discussion.

\section{Equations describing a homoclinic corner}
\label{sec:setup}
\setcounter{equation}{0}

Here we introduce coordinates and equations that capture the dynamics
near a generic homoclinic corner of the PWS map $f$, (\ref{eq:f}).
We first define linearising coordinates in a neighbourhood of a saddle fixed point,
then describe excursions away from the fixed point.
This approach has been successful for unfolding
a variety of generic and degenerate homoclinic tangencies of smooth maps,
see for instance \cite{PaTa93,GaSi72,HoWh84,GoGo09,HiLa95}\removableFootnote{
Note that almost-linearising coordinates are used in \cite{GaSi72,GoGo09}.
}.

Throughout this paper $f^n$ denotes the composition of $f$ with itself $n$ times\removableFootnote{
On pg.~305 of \cite{Ro04} there appears the phrase: ``the composition of $f$ with itself $n$ times''.
}.
Given a list of indices $\cL = j_0 \cdots j_{n-1}$, where each $j_i \in \{ 1,\ldots,m \}$,
\begin{equation}
f_{\cL} = f_{j_{n-1}} \circ \cdots \circ f_{j_0} \;,
\label{eq:fL}
\end{equation}
denotes the composition of the components of $f$ in the order specified by $\cL$.

\subsubsection*{Linearising coordinates}

We assume the homoclinic corner occurs at $\xi = 0$.
Moreover, we assume that the saddle fixed point associated with the homoclinic corner
lies at the origin and in the interior of $\cM_1$
for all sufficiently small values of $\xi$.
We let $\lambda(\xi)$ and $\sigma(\xi)$ denote the eigenvalues of the fixed point, and assume
\begin{equation}
0 < \lambda(0) < 1 < \sigma(0) \;.
\label{eq:eigs1}
\end{equation}
If instead one or both eigenvalues are negative, we can study $f^2$.
We also assume\removableFootnote{
I need to go through the calculations in the case $\lambda(0) \sigma(0) > 1$.
Are there any features that would not be apparent by considering
$\lambda(0) \sigma(0) < 1$ with $f^{-1}$?
}
\begin{equation}
\lambda(0) \sigma(0) < 1 \;.
\label{eq:eigs2}
\end{equation}
If instead $\lambda(0) \sigma(0) > 1$ and $f$ is invertible, we can study $f^{-1}$.

We assume each $f_j$ is $C^{\infty}$\removableFootnote{
We could assume $f_j$ is $C^2$, then there exists $C^1$ linearising coordinates.
This would be sufficient for the purposes of this paper,
however the error terms become distractingly complicated.
For example, the trace of the Jacobian matrix of the $k^{\rm th}$ single-round periodic solution
would be written as
\begin{equation}
\tau^{\cS}_k = \left( b^{\cS}_2 + o(1) \right) \sigma^k + \cObig \left( \lambda^k \right) \;.
\end{equation}
By assuming one extra degree of differentiability,
the $o(1)$ turns into $\cObig \left( \sigma^{-k} \right)$,
and so the $\cObig \left( \lambda^k \right)$ in the above expression is not particularly relevant.
}.
This ensures that there exists $C^{\infty}$ linearising coordinates \cite{St58,Ha64}\removableFootnote{
The smoothness of the coordinate change is a complicated subject.

Theorem 1 of \cite{St58}:
If the map is $C^K$,
and the eigenvalues are ``non-resonant'',
then there exists $C^J$ linearising coordinates, for some $J$.
Moreover if $K = \infty$ we can take $J = \infty$.
In two dimensions, if the eigenvalues satisfy $0 < \lambda < 1 < \sigma$,
then they are non-resonant.

Theorems 12.1 and 12.2 of \cite{Ha64}:
For any $J > 0$ (or $J = \infty$),
there exists $K \le 2$ (or $K = \infty$)
such that if the map is $C^K$,
and the eigenvalues are ``non-resonant'',
then there exists $C^J$ linearising coordinates.

Theorem 7 of Appendix 1 of \cite{PaTa93}
(referencing \cite{Ha64}):
In two dimensions with $0 < \lambda < 1 < \sigma$,
if the map is $C^2$ then there exists $C^1$ linearising coordinates.
Moreover if the map is $C^3$ then there exists $\ee > 0$ 
such that there exists $C^{1+\ee}$ linearising coordinates.
(On page 47 they assume the map is $C^{\infty}$ in order to 
be able to say there exists $C^2$ linearising coordinates
in order to derive the limiting quadratic map, and reference \cite{St58}.)

In \cite{GoSh96} (referencing \cite{GoSh90}):
in two dimensions with $0 < \lambda < 1 < \sigma$,
if the map is $C^K$ (with $K \ge 3$)
then there exists $C^{K+1}$ almost-linearising coordinates:
the map is almost linear (see eqn.~12)
in a way that is sufficient in that the $k^{\rm th}$ power
(at least in cross coordinates) is linear plus higher order terms (see eqn.~13).
}.
That is, there exists a $C^{\infty}$ coordinate change and an open bounded convex\removableFootnote{
We will want all of the $x$-axis from $x=0$ to $x=1$
and all of the $y$-axis from $y=0$ to $y=1$ 
to lie in the interior of $\cN$.
}
neighbourhood $\cN \subset \cM_1$ that contains the origin,
such that within $\cN$ the map $f$ is given by
\begin{equation}
f_1(x,y;\xi) = \begin{bmatrix} \lambda(\xi) x \\ \sigma(\xi) y \end{bmatrix} \;,
\label{eq:f1}
\end{equation}
for all sufficiently small values of $\xi$.
If instead we only assumed that each $f_j$ was $C^2$,
then we can only guarantee the existence of $C^1$ linearising coordinates \cite{PaTa93,Ha64}.
In this scenario the main results of this paper are unchanged
but the error terms in the key expressions are more complicated to state.

We use $W^s(0,0;\xi)$ and $W^u(0,0;\xi)$ to denote the stable and unstable manifolds of the origin for the map $f$,
and use $E^s(0,0;\xi)$ and $E^u(0,0;\xi)$ to denote the stable and unstable manifolds of the origin
for the linearisation (\ref{eq:f1}).
Notice, $E^s(0,0;\xi)$ and $E^u(0,0;\xi)$ are, respectively, the $x$ and $y$-axes minus the origin.
Since $f$ is equal to its linearisation inside $\cN$,
the parts of $W^s(0,0;\xi)$ and $W^u(0,0;\xi)$ that emanate from the origin
are equal to $E^s(0,0;\xi)$ and $E^u(0,0;\xi)$ within $\cN$.

\begin{figure}[b!]
\begin{center}
\setlength{\unitlength}{1cm}
\begin{picture}(8,6)
\put(0,0){\includegraphics[height=6cm]{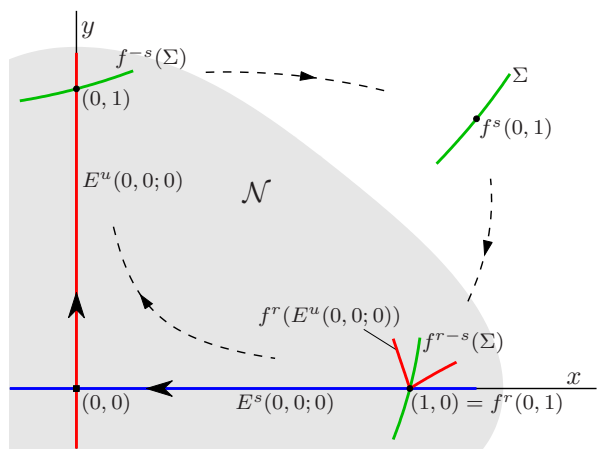}}
\put(7.4,.88){\footnotesize $x$}
\put(.95,5.55){\footnotesize $y$}
\put(3.1,3.3){\small $\cN$}
\put(.95,.54){\scriptsize $(0,0)$}
\put(.95,4.58){\scriptsize $(0,1)$}
\put(6.24,4.2){\scriptsize $f^s(0,1)$}
\put(5.32,.54){\scriptsize $(1,0) = f^r(0,1)$}
\put(1.4,5.13){\scriptsize $f^{-s}(\Sigma)$}
\put(6.68,4.9){\scriptsize $\Sigma$}
\put(5.48,1.36){\scriptsize $f^{r-s}(\Sigma)$}
\put(3.3,1.7){\scriptsize $f^r(E^u(0,0;0))$}
\put(.96,3.5){\scriptsize $E^u(0,0;0)$}
\put(3,.54){\scriptsize $E^s(0,0;0)$}
\end{picture}
\caption{
A schematic of the dynamics of $f$ for $\xi = 0$.
Inside the neighbourhood $\cN$ (the shaded region), $f$ is given by the linear map (\ref{eq:f1}).
The manifold $f^r(E^u(0,0;0))$ intersects $E^s(0,0;0)$ at $(1,0)$ forming a corner,
possibly with the orientation indicated here
(other orientations are shown in Fig.~\ref{fig:orientationHCCorner}).
\label{fig:setupHCCorner}
}
\end{center}
\end{figure}

\subsubsection*{Consequences of a homoclinic corner}

We now exploit the assumption that the origin has a homoclinic corner for $\xi = 0$.
This assumption implies that there exists a point in $E^u(0,0;0) \cap \cN$
that maps to a point in $E^s(0,0;0) \cap \cN$
under $r$ iterations of $f$, for some $r \in \mathbb{Z}^+$,
and that the manifold $f^r(E^u(0,0;0))$ forms a corner at the point in $E^s(0,0;0) \cap \cN$.
By scaling the $x$ and $y$ coordinates,
we can assume that these points are $(0,1)$ and $(1,0)$ respectively, that is,
\begin{equation}
\begin{bmatrix} 1 \\ 0 \end{bmatrix} = f^r(0,1;0) \;,
\label{eq:10}
\end{equation}
see Fig.~\ref{fig:setupHCCorner}.

Since $f^r(E^u(0,0;0))$ forms a corner at $(1,0)$,
the forward orbit of $(0,1)$ must intersect a switching manifold of $f$
before reaching $(1,0)$.
That is, there exists $s \in \{ 1,\ldots,r-1 \}$
such that $f^s(0,1;0) \in \Sigma$, where $\Sigma$ is a switching manifold of $f$.
For genericity we require $s$ to be unique,
that is $f^i(0,1;0)$ does not lie on a switching manifold of $f$
for all $i \in \{ 1,\ldots,r-1 \}$, $i \ne s$.
We also assume that $\Sigma$ is locally a $C^2$ curve\removableFootnote{
We could assume it is $C^{\infty}$ for simplicity,
but this makes the results less general,
plus there is no problem with $\Sigma$ not being $C^{\infty}$
given that the components of $f$ are $C^{\infty}$.
As a minimal example, suppose that $\Sigma$ is given by $x=0$ for $y \le 0$ 
and given by $x=y^3$ for $y \ge 0$.
Suppose that to the left of $\Sigma$, the first component of the map is given by $x(x-y^3)$,
and that to the right of $\Sigma$, the first component of the map is given by $-x(x-y^3)$.
Then $\Sigma$ is $C^2$ but not $C^3$,
and the first component of the map is continuous but not $C^1$ on $\Sigma$
(although it is $C^1$ at $(x,y) = (0,0)$).
},
and write
\begin{equation}
\Sigma = \left\{ (x,y) \,\middle|\; h(x,y;\xi) = 0 \right\} \;,
\label{eq:Sigma}
\end{equation}
for a $C^2$ function $h$.
For $(x,y)$ near $(0,1)$ and small values of $\xi$, we can therefore write
\begin{equation}
f^r(x,y;\xi) = \begin{cases}
f_{\cX}(x,y;\xi) \;, & h(f^s(x,y;\xi);\xi) \le 0 \\
f_{\cY}(x,y;\xi) \;, & h(f^s(x,y;\xi);\xi) \ge 0
\end{cases} \;,
\label{eq:fr}
\end{equation}
where $\cX$ and $\cY$ are lists of indices of length $r$
that differ only in index $s$.

Given $\xi$, the set of all $(x,y)$ near $(0,1)$ for which $h(f^s(x,y;\xi);\xi) = 0$
is $f^{-s}(\Sigma)$ -- the inverse of $\Sigma$ under $f^s$,
or the part of the inverse of $\Sigma$ that lies near $(0,1)$ in the case that $f^s$ has multiple inverses.
For genericity we require $f^{-s}(\Sigma)$ to be a curve
transversally intersecting $E^u(0,0;0)$ at $(0,1)$ for $\xi = 0$,
as shown in Fig.~\ref{fig:setupHCCorner}.
Since $f$ is $C^\infty$ and $h$ is $C^2$, we can write
\begin{equation}
h(f^s(x,y;\xi);\xi) = p_1 x + p_2 (y-1) + p_3 \xi + \cObig \left( \|x,y-1,\xi\|^2 \right) \;,
\label{eq:hfs}
\end{equation}
for some constants $p_1, p_2, p_3 \in \mathbb{R}$,
where $\cO$ denotes ``big-O notation''.
The transversal intersection is ensured by assuming $p_2 \ne 0$.
Moreover, by scaling $h$ we can assume
\begin{equation}
p_2 = 1 \;.
\label{eq:p21}
\end{equation}
Then by (\ref{eq:hfs}) and (\ref{eq:p21}), we can write
\begin{equation}
f^{-s}(\Sigma) = \left\{ (x,y) \,\middle|\; y = g(x;\xi) \right\} \;,
\label{eq:Sigmams}
\end{equation}
where $g$ is a $C^2$ function satisfying
\begin{equation}
g(x;\xi) = 1 - p_1 x - p_3 \xi + \cObig \left( \|x,\xi\|^2 \right) \;.
\label{eq:g}
\end{equation}
We can also rewrite (\ref{eq:fr}) as
\begin{equation}
f^r(x,y;\xi) = \begin{cases}
f_{\cX}(x,y;\xi) \;, & y \le g(x;\xi) \\
f_{\cY}(x,y;\xi) \;, & y \ge g(x;\xi)
\end{cases} \;.
\label{eq:fr2}
\end{equation}

To study the behaviour of $f_{\cX}$ and $f_{\cY}$
near $(x,y) = (0,1)$ and for small values of $\xi$,
we expand $f_{\cX}$ and $f_{\cY}$ about $(x,y;\xi) = (0,1;0)$.
Since $f$ is continuous, $f_{\cX}$ and $f_{\cY}$ are equal on $f^{-s}(\Sigma)$.
It follows that there exist constants
$a_1, a_2, b^{\cX}_1, b^{\cX}_2, b^{\cY}_1, b^{\cY}_2, c_1, c_2 \in \mathbb{R}$ such that
\begin{equation}
f_{\cS}(x,y;\xi) =
\begin{bmatrix}
1 + \left( a_1 + b^{\cS}_1 p_1 \right) x +
b^{\cS}_1 (y-1) + \left( c_1 + b^{\cS}_1 p_3 \right) \xi \\
\left( a_2 + b^{\cS}_2 p_1 \right) x +
b^{\cS}_2 (y-1) + \left( c_2 + b^{\cS}_2 p_3 \right) \xi
\end{bmatrix} + \cObig \left( \|x,y-1,\xi\|^2 \right) \;,
\label{eq:fS}
\end{equation}
where throughout the paper
we use $\cS$ to represent either $\cX$ or $\cY$.

The determinants of the Jacobian matrices of $f_{\cX}$ and $f_{\cY}$ are given from (\ref{eq:fS}) by
\begin{equation}
\det \left( D f_{\cS} \right)|_{(0,1;0)} =
a_1 b^{\cS}_2 - a_2 b^{\cS}_1 \;,
\label{eq:detDfS}
\end{equation}
for $\cS = \cX, \cY$.
If the product
$\det \left( D f_{\cX} \right)|_{(0,1;0)} \det \left( D f_{\cY} \right)|_{(0,1;0)}$
is positive, then $f^r$ is locally invertible.
If $\det \left( D f_{\cX} \right)|_{(0,1;0)}$ and $\det \left( D f_{\cY} \right)|_{(0,1;0)}$
are both positive, then $f^r$ is also locally orientation-preserving\removableFootnote{
The orientation-reversing case seems rather exotic
given that $f_1$ is orientation-preserving.
}.

\section{Three genericity assumptions}
\label{sec:assumptions}
\setcounter{equation}{0}

In this section we identify three additional conditions
necessary to have a generic homoclinic corner.

\subsubsection*{A transversal intersection between $f^{r-s}(\Sigma)$ and $E^s$}

As explained above, $f^{-s}(\Sigma)$ is a curve
transversally intersecting $E^u(0,0;0)$ at $(0,1)$ for $\xi = 0$, Fig.~\ref{fig:setupHCCorner}.
For genericity we also require that $f^{r-s}(\Sigma)$ is a curve
transversally intersecting $E^s(0,0;0)$ at $(1,0)$ for $\xi = 0$.
By (\ref{eq:fr2}) and (\ref{eq:fS}), points on $f^{r-s}(\Sigma)$ for $\xi = 0$ are given by
\begin{equation}
f^r(x,g(x;0);0) = \begin{bmatrix}
1 + a_1 x \\ a_2 x
\end{bmatrix} + \cObig \left( |x|^2 \right) \;.
\label{eq:frmsSigma}
\end{equation}
Thus the condition
\begin{equation}
a_2 \ne 0 \;,
\label{eq:a2nonzero}
\end{equation}
ensures $f^{r-s}(\Sigma)$ transversally intersects $E^s(0,0;0)$.

\subsubsection*{A sign condition ensuring a homoclinic corner}

We have assumed that $f^r(E^u(0,0;0))$ forms a corner at $(1,0) \in E^s(0,0;0)$.
That is, locally, $f^r(E^u(0,0;0))$ includes the point $(1,0)$
and otherwise either lies entirely above $E^s(0,0;0)$ or below $E^s(0,0;0)$.
By (\ref{eq:fr2}) and (\ref{eq:fS}), points on $f^r(E^u(0,0;0))$ are given by
\begin{equation}
f^r(0,y;0) = \begin{cases}
\begin{bmatrix} 1 + b^{\cX}_1 (y-1) \\ b^{\cX}_2 (y-1) \end{bmatrix}
+ \cObig \left( |y-1|^2 \right) \;, & y \le 1 \\
\begin{bmatrix} 1 + b^{\cY}_1 (y-1) \\ b^{\cY}_2 (y-1) \end{bmatrix}
+ \cObig \left( |y-1|^2 \right) \;, & y \ge 1
\end{cases} \;.
\label{eq:frEu}
\end{equation}
Thus the condition
\begin{equation}
b^{\cX}_2 b^{\cY}_2 < 0 \;,
\label{eq:bX2bY2negative}
\end{equation}
ensures that $f^r(E^u(0,0;0))$ intersects $E^s(0,0;0)$ at a corner.

\begin{figure}[t!]
\begin{center}
\setlength{\unitlength}{1cm}
\begin{picture}(10,6.6)
\put(1.5,3.2){\includegraphics[height=3cm]{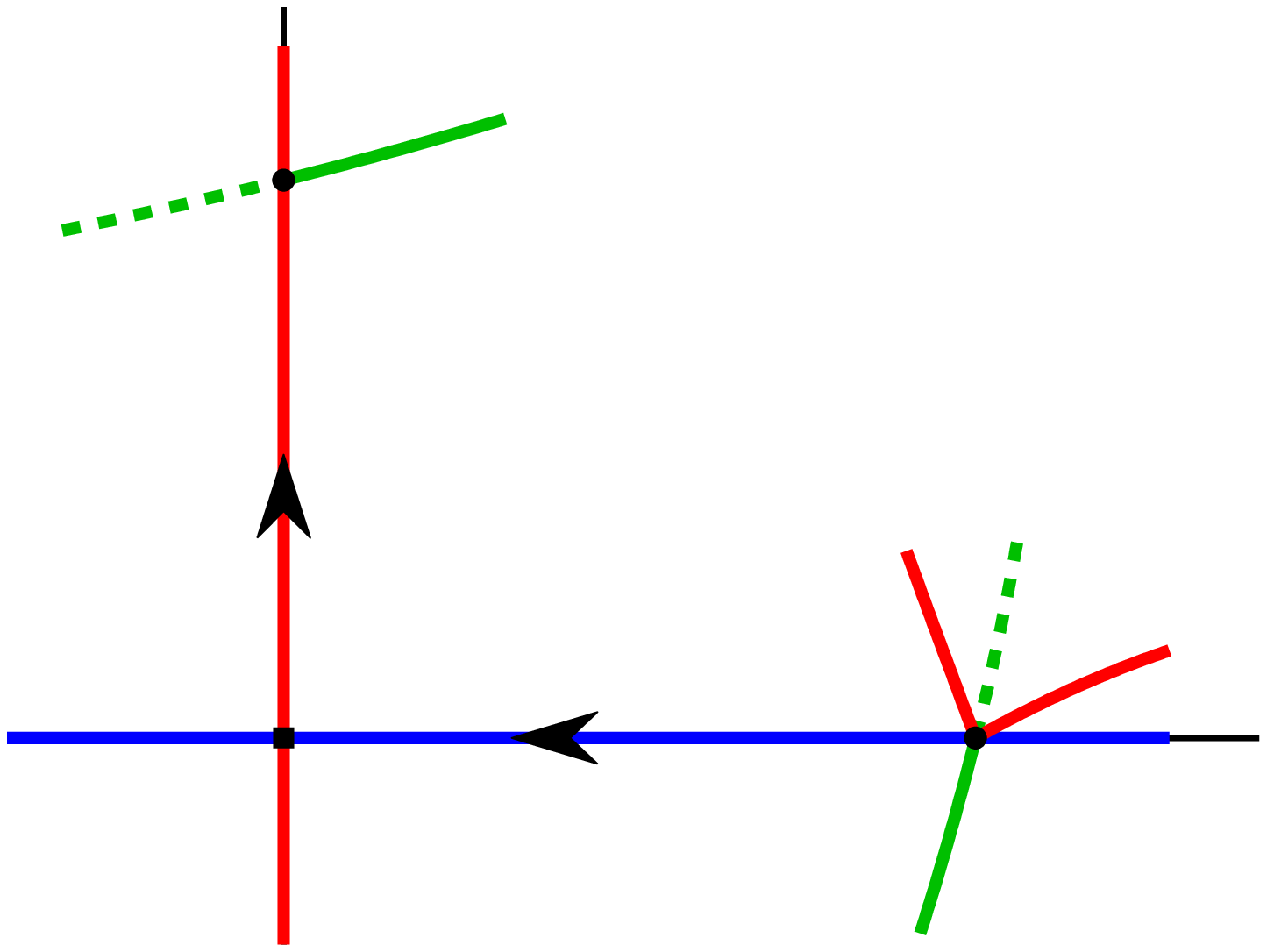}}
\put(6,3.2){\includegraphics[height=3cm]{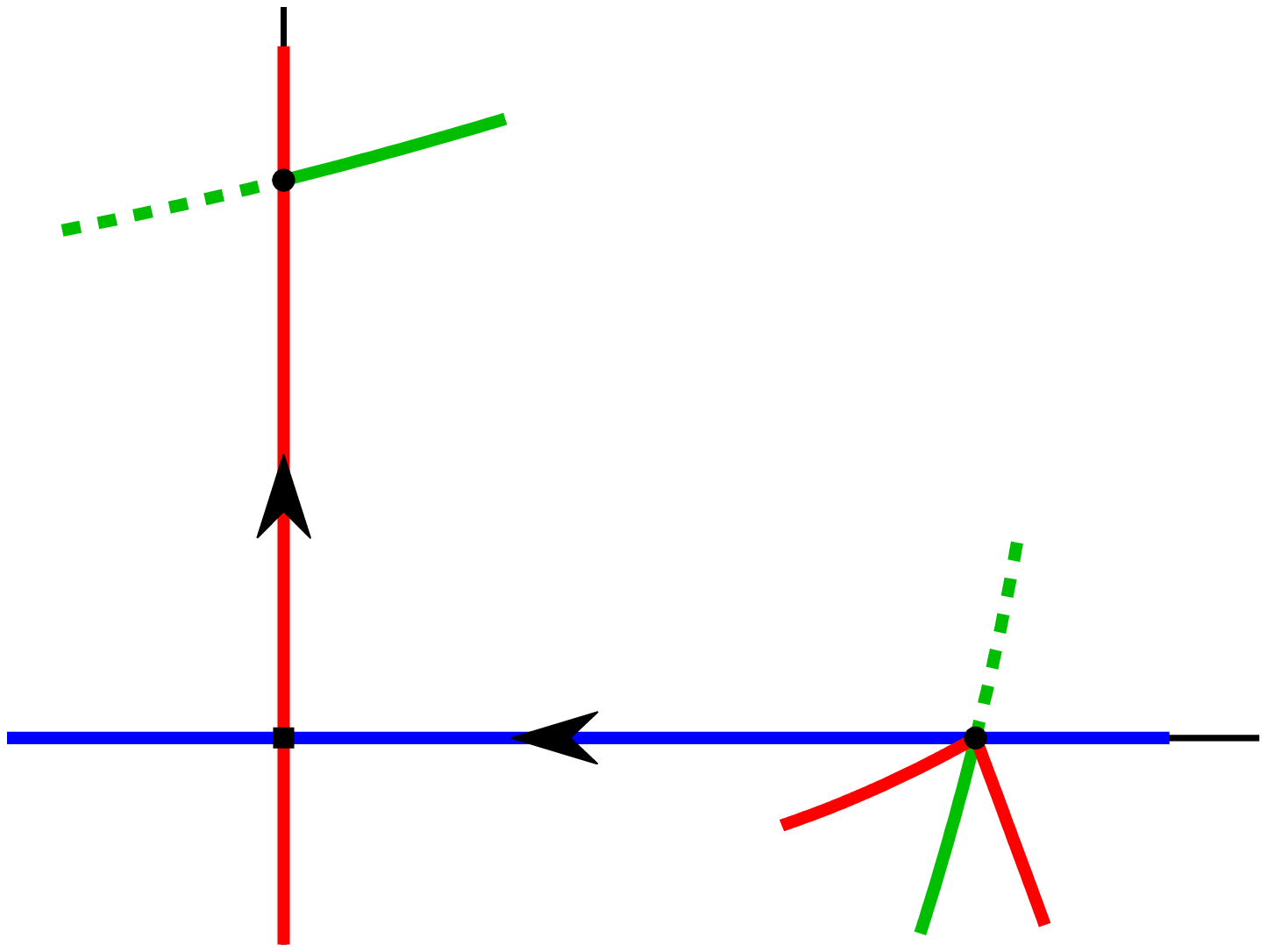}}
\put(1.5,0){\includegraphics[height=3cm]{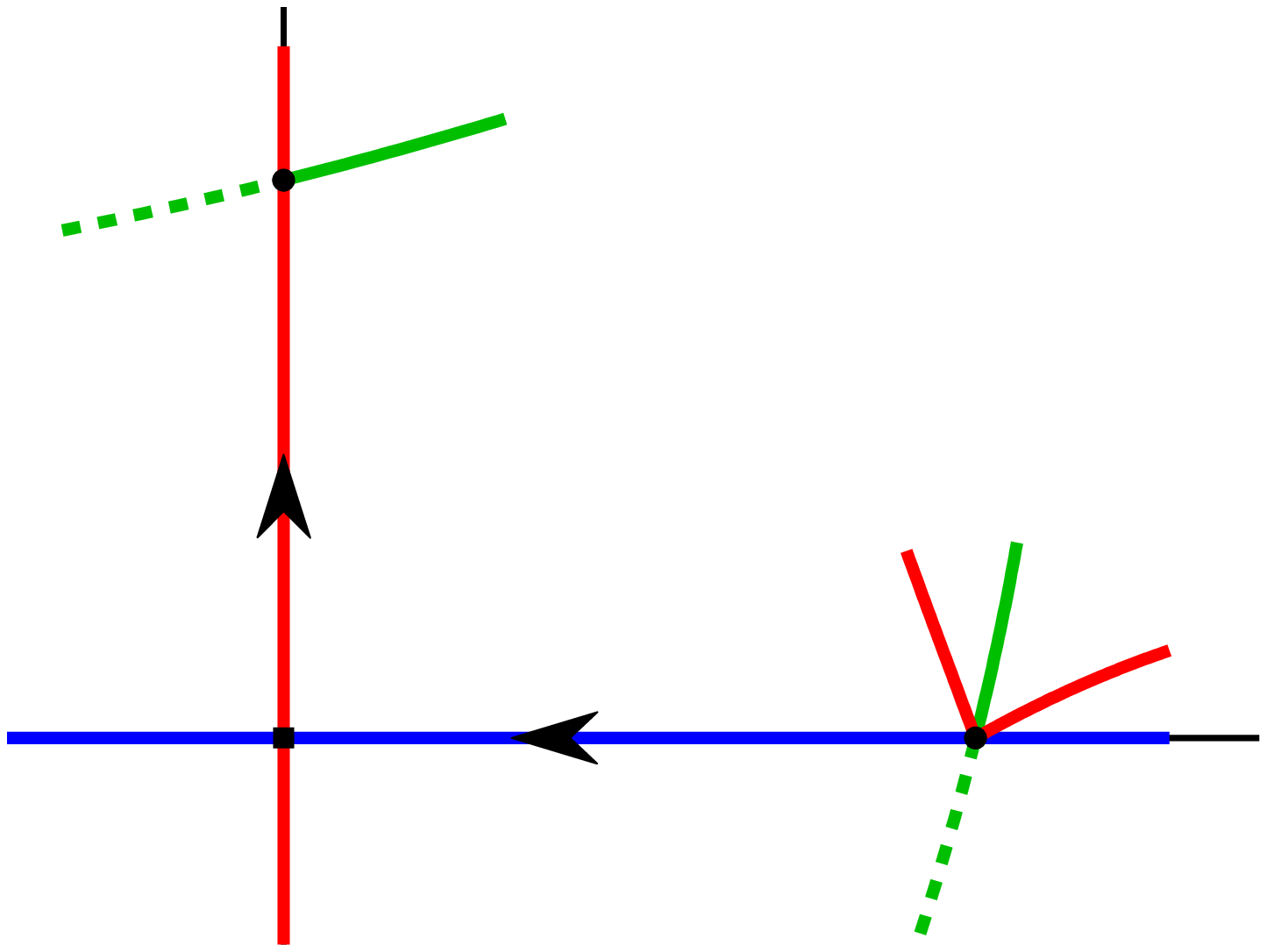}}
\put(6,0){\includegraphics[height=3cm]{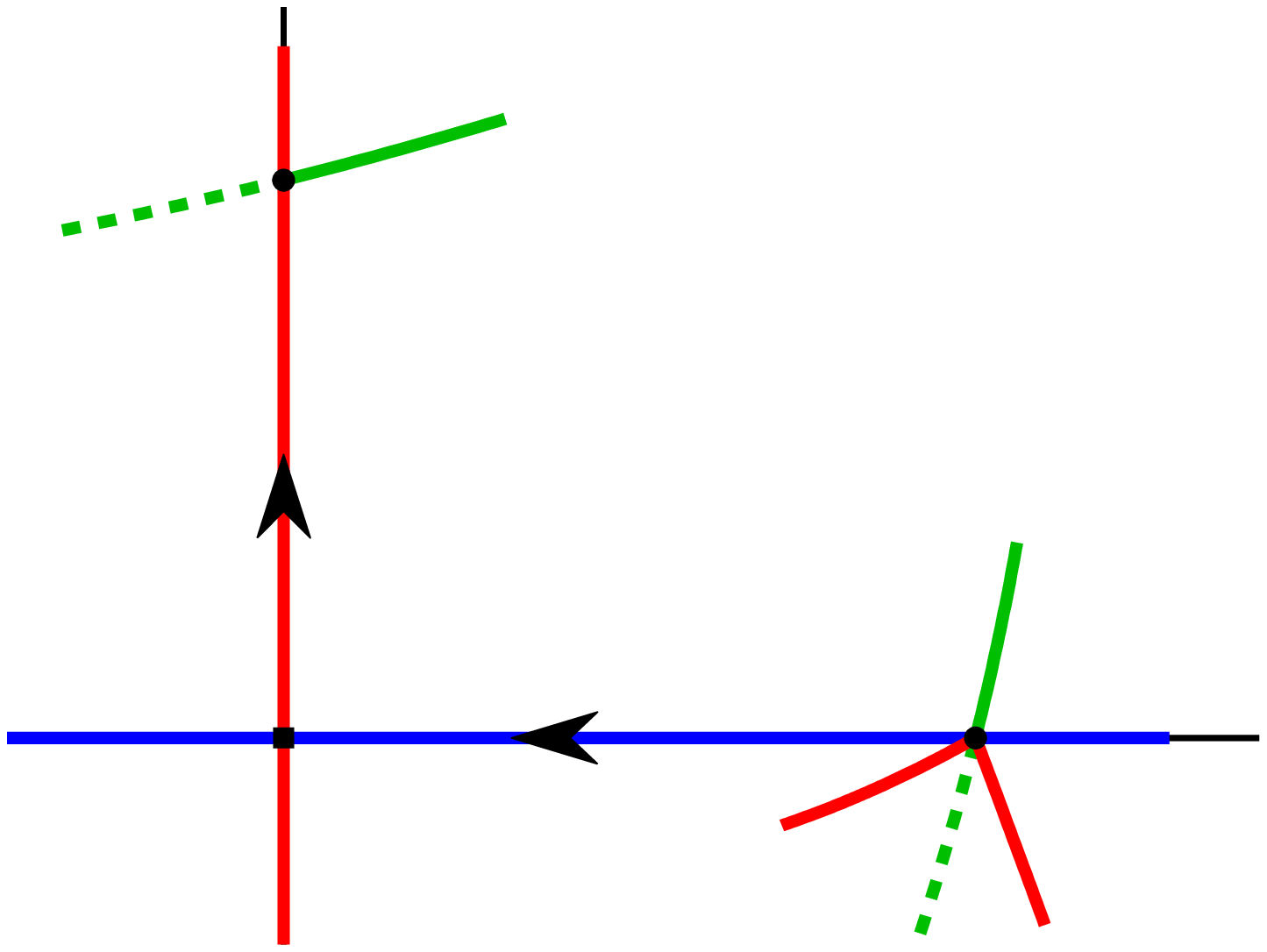}}
\put(0,4.95){\small $a_2 < 0$}
\put(0,1.75){\small $a_2 > 0$}
\put(3.1,6.4){\small $b^{\cX}_2 < 0$}
\put(7.6,6.4){\small $b^{\cX}_2 > 0$}
\put(5.21,3.91){\scriptsize $x$}
\put(2.43,5.96){\scriptsize $y$}
\put(2.4,3.64){\tiny $(0\hspace{-.1mm},\hspace{-.4mm}0)$}
\put(2.4,5.39){\tiny $(0\hspace{-.1mm},\hspace{-.6mm}1)$}
\put(4.52,3.64){\tiny $(1\hspace{-.2mm},\hspace{-.4mm}0)$}
\put(9.71,3.91){\scriptsize $x$}
\put(6.93,5.96){\scriptsize $y$}
\put(6.9,3.64){\tiny $(0\hspace{-.1mm},\hspace{-.4mm}0)$}
\put(6.9,5.39){\tiny $(0\hspace{-.1mm},\hspace{-.6mm}1)$}
\put(9.12,3.64){\tiny $(1\hspace{-.2mm},\hspace{-.4mm}0)$}
\put(5.21,.71){\scriptsize $x$}
\put(2.43,2.76){\scriptsize $y$}
\put(2.4,.44){\tiny $(0\hspace{-.1mm},\hspace{-.4mm}0)$}
\put(2.4,2.19){\tiny $(0\hspace{-.1mm},\hspace{-.6mm}1)$}
\put(4.52,.44){\tiny $(1\hspace{-.2mm},\hspace{-.4mm}0)$}
\put(9.71,.71){\scriptsize $x$}
\put(6.93,2.76){\scriptsize $y$}
\put(6.9,.44){\tiny $(0\hspace{-.1mm},\hspace{-.4mm}0)$}
\put(6.9,2.19){\tiny $(0\hspace{-.1mm},\hspace{-.6mm}1)$}
\put(9.12,.44){\tiny $(1\hspace{-.2mm},\hspace{-.4mm}0)$}
\end{picture}
\caption{
Schematics showing the orientation of
$f^{r-s}(\Sigma)$ and $f^r(E^u(0,0;0))$ near $(1,0)$ for $\xi = 0$,
as governed by the signs of $a_2$ and $b^{\cX}_2$.
Under $f^r$, the dashed [solid] part of $f^{-s}(\Sigma)$
maps to the dashed [solid] part of $f^{r-s}(\Sigma)$.
\label{fig:orientationHCCorner}
}
\end{center}
\end{figure}

By (\ref{eq:frmsSigma}) and (\ref{eq:frEu}),
the orientation of $f^{r-s}(\Sigma)$ and $f^r(E^u(0,0;0))$ near $(1,0)$ for $\xi = 0$
is governed by the signs of $a_2$ and $b^{\cX}_2$
(or equivalently $b^{\cY}_2$, as we now assume (\ref{eq:bX2bY2negative}) holds).
This is summarised by Fig.~\ref{fig:orientationHCCorner}.
The part of $f^{-s}(\Sigma)$ near $(0,1)$ that lies to the right of the $y$-axis
maps to the part of $f^{r-s}(\Sigma)$ near $(1,0)$ below the $x$-axis if $a_2 < 0$,
and to the part of $f^{r-s}(\Sigma)$ near $(1,0)$ above the $x$-axis if $a_2 > 0$.
Near $(1,0)$, the manifold $f^r(E^u(0,0;0))$ lies above the $x$-axis if $b^{\cX}_2 < 0$,
and below the $x$-axis if $b^{\cX}_2 > 0$.

\subsubsection*{A condition ensuring a generic unfolding}

Lastly we need to ensure that $\xi$ unfolds the homoclinic corner in a generic fashion.
That is, we require that the distance from the corner of $f^r(E^u(0,0;\xi))$ near $(1,0)$
to $E^s(0,0;\xi)$ is a linear function of $\xi$, to leading order.
The corner is located at the image of $E^u(0,0;\xi) \cap f^{-s}(\Sigma)$ under $f^r$.
The intersection $E^u(0,0;\xi) \cap f^{-s}(\Sigma)$ is given by $(x,y) = (0,g(0;\xi))$
and by (\ref{eq:fS}) its image under $f^r$ is equal to
\begin{equation}
f^r(0,g(0;\xi);\xi) = \begin{bmatrix} 1 + c_1 \xi \\ c_2 \xi \end{bmatrix}
+ \cObig \left( |\xi|^2 \right) \;.
\end{equation}
Hence the condition
\begin{equation}
c_2 \ne 0 \;,
\label{eq:c2nonzero}
\end{equation}
ensures that $\xi$ is a generic unfolding parameter\removableFootnote{
We could scale $\xi$ such that $c_2 = 1$,
but I don't think this is actually that helpful.
}.

\section{Periodic solutions}
\label{sec:unstable}
\setcounter{equation}{0}

In this section we identify single-round periodic solutions,
then show that any periodic solution near the homoclinic corner is unstable.
Let us first summarise the assumptions given above.
Inside the neighbourhood $\cN \subset \cM_1$, which contains the points $(1,0)$ and $(0,1)$,
the map $f$ is given by (\ref{eq:f1}) where $\lambda$ and $\sigma$ satisfy (\ref{eq:eigs1}) and (\ref{eq:eigs2}).
Near $(0,1)$, the map $f^r$ is given by (\ref{eq:fr2}),
where $g$ is given by (\ref{eq:g}) and $f_{\cX}$ and $f_{\cY}$ are given by (\ref{eq:fS}).
The parameters of (\ref{eq:fS}) satisfy the genericity conditions
(\ref{eq:a2nonzero}), (\ref{eq:bX2bY2negative}) and (\ref{eq:c2nonzero}).

\subsubsection*{Single-round periodic solutions}

Given $(x,y;\xi)$ near $(0,1;0)$, the point $f^r(x,y;\xi)$ lies near $(1,0)$.
Therefore, since the forward orbit of $(1,0)$
limits to $(0,0)$ without escaping $\cN$,
we have $f^r(x,y;\xi), f^{r+1}(x,y;\xi), \ldots, f^{r+k-1}(x,y;\xi) \in \cN$,
for some $k \in \mathbb{Z}^+$, in which case
\begin{equation}
f^{r+k}(x,y;\xi) = \begin{cases}
f_{\cX 1^k}(x,y;\xi) \;, & y \le g(x;\xi) \\
f_{\cY 1^k}(x,y;\xi) \;, & y \ge g(x;\xi)
\end{cases} \;.
\label{eq:frpk}
\end{equation}
Moreover, the value of $k$ can be made as large as we like
by choosing $(x,y;\xi)$ arbitrarily close to $(0,1;0)$.

A fixed point of (\ref{eq:frpk}) is one point of a single-round periodic solution of $f$
of period $n = r+k$.
By combining (\ref{eq:f1}) and (\ref{eq:fS}), we obtain
\begin{equation}
f_{\cS 1^k}(x,y;\xi) = \begin{bmatrix}
\lambda^k \left\{ 1 + \left( a_1 + b^{\cS}_1 p_1 \right) x +
b^{\cS}_1 (y-1) + \left( c_1 + b^{\cS}_1 p_3 \right) \xi + \cObig \left( \|x,y-1,\xi\|^2 \right) \right\} \\
\sigma^k \left\{ \left( a_2 + b^{\cS}_2 p_1 \right) x +
b^{\cS}_2 (y-1) + \left( c_2 + b^{\cS}_2 p_3 \right) \xi + \cObig \left( \|x,y-1,\xi\|^2 \right) \right\}
\end{bmatrix} \;,
\label{eq:fS1k}
\end{equation}
for $\cS = \cX, \cY$.
With $\xi = \cObig \left( \sigma^{-k} \right)$,
the functions $f_{\cX 1^k}$ and $f_{\cY 1^k}$ each have a unique fixed point near $(0,1)$ satisfying
\begin{equation}
x^{\cS}_k(\xi) = \lambda^k + \cObig \left( \sigma^{-2k} \right) \;, \qquad
y^{\cS}_k(\xi) = 1 + \frac{\sigma^{-k} - \left( c_2 + b^{\cS}_2 p_3 \right) \xi}{b^{\cS}_2} +
\cObig \left( \lambda^k \right) + \cObig \left( \sigma^{-2k} \right) \;,
\label{eq:xyS1k}
\end{equation}
where we have omitted the dependence of $\lambda$ and $\sigma$ on $\xi$ for clarity.

If $y^{\cX}_k(\xi) \le g \left( x^{\cX}_k(\xi);\xi \right)$,
then $\left( x^{\cX}_k(\xi), y^{\cX}_k(\xi) \right)$
is an admissible (i.e.~valid) fixed point of (\ref{eq:frpk}).
If $k$ is also sufficiently large, then the forward orbit of $\left( x^{\cX}_k(\xi), y^{\cX}_k(\xi) \right)$
is an admissible periodic solution of $f$.
Similarly, if $y^{\cY}_k(\xi) \ge g \left( x^{\cY}_k(\xi);\xi \right)$ and $k$ is sufficiently large,
then the forward orbit of $\left( x^{\cY}_k(\xi), y^{\cY}_k(\xi) \right)$
is an admissible periodic solution of $f$.

The Jacobian matrix of $f_{\cS 1^k}$ at $\left( x^{\cS}_k(\xi), y^{\cS}_k(\xi) \right)$ is
\begin{equation}
\left( D f_{\cS 1^k} \right) \big|_{\left( x^{\cS}_k(\xi), y^{\cS}_k(\xi); \xi \right)} =
\begin{bmatrix}
\left( a_1 + b^{\cS}_1 p_1 \right) \lambda^k + \cObig \left( \lambda^k \sigma^{-k} \right) &
b^{\cS}_1 \lambda^k + \cObig \left( \lambda^k \sigma^{-k} \right)\\
\left( a_2 + b^{\cS}_2 p_1 \right) \sigma^k + \cO(1) &
b^{\cS}_2 \sigma^k + \cO(1)
\end{bmatrix} \;.
\label{eq:DfS1k}
\end{equation}
We let $\tau^{\cS}_k$ and $\delta^{\cS}_k$ denote the trace and determinant of (\ref{eq:DfS1k}):
\begin{align}
\tau^{\cS}_k = {\rm trace} \left( \left( D f_{\cS 1^k} \right)
\big|_{\left( x^{\cS}_k(\xi), y^{\cS}_k(\xi); \xi \right)} \right) &=
b^{\cS}_2 \sigma^k + \cO(1) \;, \label{eq:tauS1k} \\
\delta^{\cS}_k = {\rm det} \left( \left( D f_{\cS 1^k} \right)
\big|_{\left( x^{\cS}_k(\xi), y^{\cS}_k(\xi); \xi \right)} \right) &=
\left( a_1 b^{\cS}_2 - a_2 b^{\cS}_1 \right) \lambda^k \sigma^k +
\cObig \left( \lambda^k \right) \;. \label{eq:deltaS1k}
\end{align}
Since $\sigma > 1$ and $b^{\cX}_2$ and $b^{\cY}_2$ are nonzero, $|\tau^{\cS}_k| \to \infty$ as $k \to \infty$.
Also $\lambda \sigma < 1$, thus $\delta^{\cS}_k \to 0$ as $k \to \infty$.
Therefore the Jacobian matrix (\ref{eq:DfS1k}) has one eigenvalue with modulus less than $1$,
and one eigenvalue with modulus greater than $1$.
Hence the forward orbits of $\left( x^{\cX}_k(\xi), y^{\cX}_k(\xi) \right)$
and $\left( x^{\cY}_k(\xi), y^{\cY}_k(\xi) \right)$ are saddle periodic solutions
for sufficiently large values of $k$.

\subsubsection*{Multi-round periodic solutions and a lack of stability}

We now consider multi-round periodic solutions.
Such solutions involve several excursions far from $(x,y) = (0,0)$.
We anticipate the eigenvalues of a multi-round periodic solution
to be those of a product of matrices of the form (\ref{eq:DfS1k}).
The trace of such a product is dominated its $(2,2)$-element,
which has a term with powers of $\sigma$ accumulated from each matrix in the product.
With a large trace, the eigenvalues of the product cannot both have modulus less than $1$,
in which case the multi-round solution is unstable.
Below we add rigour to this argument to produce Theorem \ref{th:unstable}.

First we clarify multi-round periodic solutions\removableFootnote{
We cannot just say that a multi-round periodic solution
has an itinerary of the form
$\cS_1 1^{k_1} \cdots \cS_q 1^{k_q}$
because the points of the solution could go close to the other branches
of $W^s(0,0;0)$ and $W^u(0,0;0)$, or could conceivably be associated with a
different saddle fixed point of $f_1$.

In some sense we are being a bit restrictive by saying that excursions must be of length $r$,
but it should be understood that we only care about orbits that pass
close to $(0,1)$ and hence, matching our discussion above,
excursions can assumed to be of length $r$.
}.
Let $\cN_1$ denote the intersection of $\cN$ with the first quadrant ($x,y > 0$).
We say that an ``excursion'' consists of $r$ points, at least one of which lies outside $\cN_1$.
A $q$-round periodic solution of $f$ consists of
$q$ excursions, between each of which is a sequence of points in $\cN_1$.

\begin{theorem}
Suppose\removableFootnote{
We use $1 < y < \sigma$ because even though $(0,y)$ may not belong to $\cN$,
the point $\left( 0, \frac{y}{\sigma} \right)$ does belong to $\cN$,
hence $(0,y)$ belongs to the ``local'' part of $W^u(0,0;0)$.
Moreover, for all $\lambda < x < 1$, the point $(x,0)$ belongs to $\cN$
and hence belongs to the ``local'' part of $W^s(0,0;0)$.
}
\begin{equation}
\left\{ f^r(0,y;0) \,\middle|\; 1 < y < \sigma \right\} \cap E^s(0,0;0) = \varnothing \;.
\label{eq:noExtraIntersections}
\end{equation}
Then there exists $k \in \mathbb{Z}^+$ and $\ee > 0$,
such that for any $q \in \mathbb{Z}^+$ and any $|\xi| < \ee$,
any $q$-round periodic solution of $f$ with at least $k$ points between each excursion is unstable.
\label{th:unstable}
\end{theorem}

First let us explain condition (\ref{eq:noExtraIntersections}).
The points $(0,1)$ and $(0,\sigma) = f(0,1;0)$ both belong to $E^u(0,0;0)$
and map to $E^s(0,0;0)$ under $f^r$.
The condition (\ref{eq:noExtraIntersections}) ensures that no points on $E^u(0,0;0)$
between $(0,1)$ and $(0,\sigma)$ also map to $E^s(0,0;0)$ under $f^r$.
That is, $f^r(E^u(0,0;0))$ and $E^s(0,0;0)$ have no extra intersections such as a tangency
that could generate stable periodic solutions.

\begin{proof}
For any $M \in \mathbb{R}$, there exists $\ee > 0$ such that for all $|\xi| < \ee$
the image under $f^r$ of any point within $M \ee$ of $(0,1)$ is given by (\ref{eq:fr2})\removableFootnote{
We start with this particular statement because we need to choose $\ee$ before we choose $\xi$.
Also we need the $M$ because below we only show that $(x_i,y_i)$ is an $\cO(\ee)$ distance from $(0,1)$.
From this we conclude that there exists $M$ such that $(x_i,y_i)$ is within a distance $M \ee$ of $(0,1)$.
}.
Then given $k, q \in \mathbb{Z}^+$ and $|\xi| < \ee$,
let $\Gamma$ be a $q$-round periodic solution of $f$ with at least $k$ points between each excursion.
Let $(x_1,y_1), \ldots, (x_q,y_q)$ be the starting points of the excursions of $\Gamma$.
That is, for each $i = 1,\ldots,q$,
we have $(x_i,y_i) \in \cN_1$ and
$f^{r + k_i}(x_i,y_i;\xi) = \left( x_{i+1}, y_{i+1} \right)$,
where $k_i \ge k$ is the number of points of $\Gamma$ between the $i^{\rm th}$ excursion and
the subsequent excursion, and $(x_{q+1},y_{q+1}) = (x_1,y_1)$.

For each $i = 1,\ldots,q$, the $k$ points of $\Gamma$ prior to $(x_i,y_i)$ lie in $\cN_1$.
Since $\cN_1$ is bounded, by (\ref{eq:f1}) the point $(x_i,y_i)$ must be an $\cObig \left( \lambda^k \right)$
distance from $E^u(0,0;\xi)$.
Therefore since $(x_i,y_i) \in \cN_1$, the point $f^r(x_i,y_i)$ is an $\cObig \left( \lambda^k \right)$
distance from $f^r \left( E^u(0,0;\xi) \cap \cN_1; \xi \right)$\removableFootnote{
We need the intersect $\cN_1$ here because we do not know
how points on $E^u(0,0;\xi)$ outside $\cN_1$ iterate under $f^r$.
}.
The $k$ points of $\Gamma$ after $f^r(x_i,y_i)$ lie in $\cN_1$,
thus by (\ref{eq:f1}) the point $f^r(x_i,y_i)$ must also be an $\cObig \left( \sigma^{-k} \right)$
distance from $E^s(0,0;\xi)$.

By (\ref{eq:noExtraIntersections}) the only intersections of 
$f^r \left( E^u(0,0;0) \cap \cN_1; 0 \right)$ and $E^s(0,0;0)$
are at $(1,0)$ and a finite number of images and preimages of $(1,0)$ under (\ref{eq:f1})\removableFootnote{
This is because $\cN_1$ is bounded,
and so the number of consecutive iterates in $\cN_1$ is at least $k_i - \ell$,
where $\ell$ is a finite number related to the size of $\cN_1$
(so we could replace $k$ with $k + \ell$).
}.
Since we are dealing with large values of $k$,
we can assume that each $f^r(x_i,y_i)$ lies near $(1,0)$.
More specifically, with $\sigma^{-k} < \ee$
we have that $f^r(x_i,y_i)$ is an $\cO(\ee)$ distance from $(1,0)$.
Thus $(x_i,y_i)$ is an $\cO(\ee)$ distance from $(0,1)$
and we can assume $M$ is large enough that $(x_i,y_i)$ is within a distance $M \ee$ of $(0,1)$.

Thus each $f^{r+k_i}(x_i,y_i;\xi)$ is given by (\ref{eq:frpk}),
that is equal to $f_{\cS_i 1^{k_i}}(x_i,y_i;\xi)$,
where each $\cS_i$ is either $\cX$ or $\cY$.
Hence the eigenvalues of $\Gamma$ are those of the product
\begin{equation}
\big( D f_{\cS_q 1^{k_q}} \big) \big|_{(x_q,y_q;\xi)} \cdots
\big( D f_{\cS_1 1^{k_1}} \big) \big|_{(x_1,y_1;\xi)} \;.
\end{equation}
By (\ref{eq:DfS1k}), the trace of this product is
\begin{equation}
b^{\cS_1}_2 \cdots b^{\cS_q}_2
\sigma^{k_1 + \cdots + k_q} + \cObig \left( \lambda^k \sigma^{(q-1)k} \right) \;.
\end{equation}
Since $\sigma > 1$ and $b^{\cX}_2$ and $b^{\cY}_2$ are nonzero,
for sufficiently large values of $k$ the trace is large and thus $\Gamma$ is unstable.
\end{proof}

\section{Border-collision bifurcations of single-round periodic solutions}
\label{sec:bcbs}
\setcounter{equation}{0}

Here we identify BCBs of single-round periodic solutions.
We then change coordinates to obtain the border-collision normal form
which we use to study the stable and unstable manifolds of the periodic solutions.

\subsubsection*{The location and basic properties of the border-collision bifurcations}

As described above, for sufficiently large values of $k$
the points $\left( x^{\cX}_k, y^{\cX}_k \right)$ and $\left( x^{\cY}_k, y^{\cY}_k \right)$,
given by (\ref{eq:xyS1k}),
are fixed points of $f_{\cX 1^k}$ and $f_{\cY 1^k}$,
and the forward orbits of these points are single-round periodic solutions of $f$ with period $n = k+r$.

\begin{theorem}
For sufficiently large values of $k \in \mathbb{Z}^+$,
single-round periodic solutions of period $n = k+r$ collide in a BCB at $\xi = \xi_k$, where
\begin{equation}
\xi_k = \frac{1}{c_2} \sigma^{-k} +
\cObig \left( \lambda^k \right) + \cObig \left( \sigma^{-2k} \right) \;.
\label{eq:xik}
\end{equation}
Both periodic solutions are admissible for values of $\xi$ sufficiently close to $\xi_k$ with
\begin{equation}
{\rm sgn}(\xi - \xi_k) = {\rm sgn} \left( b^{\cX}_2 c_2 \right) \;.
\label{eq:xyS1kadmissible}
\end{equation}
The eigenvalues of the periodic solutions are
\begin{equation}
\gamma^{\cS}_{k,u} = b^{\cS}_2 \sigma^k + \cO(1) \;, \qquad
\gamma^{\cS}_{k,s} = \frac{a_1 b^{\cS}_2 - a_2 b^{\cS}_1}{b^{\cS}_2} \lambda^k +
\cObig \left( \lambda^k \sigma^{-k} \right) \;,
\label{eq:gammaus}
\end{equation}
for $\cS = \cX, \cY$.
\label{th:bcbs}
\end{theorem}

\begin{figure}[b!]
\begin{center}
\setlength{\unitlength}{1cm}
\begin{picture}(11,6.7)
\put(1.5,3.2){\includegraphics[height=3cm]{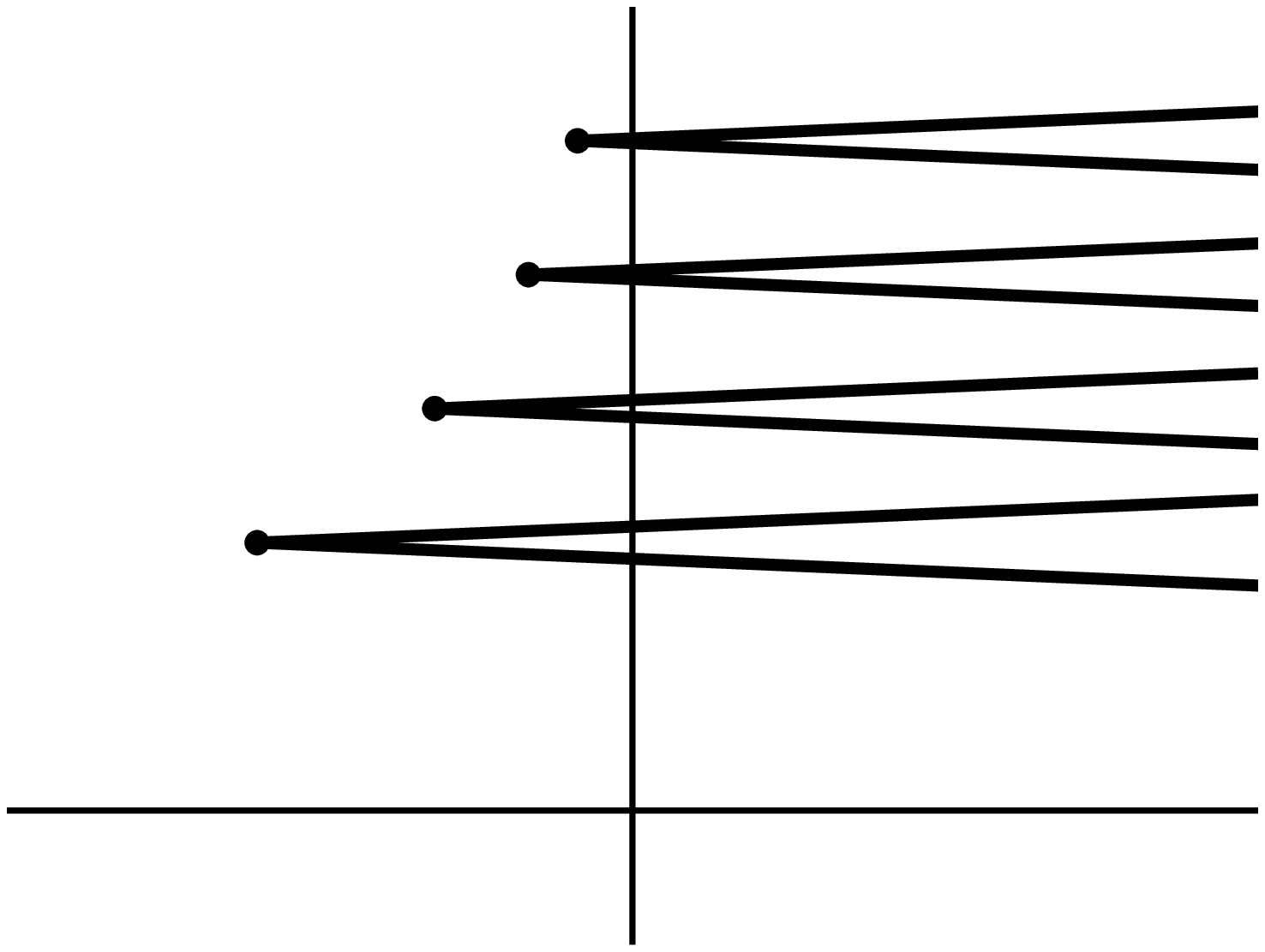}}
\put(7,3.2){\includegraphics[height=3cm]{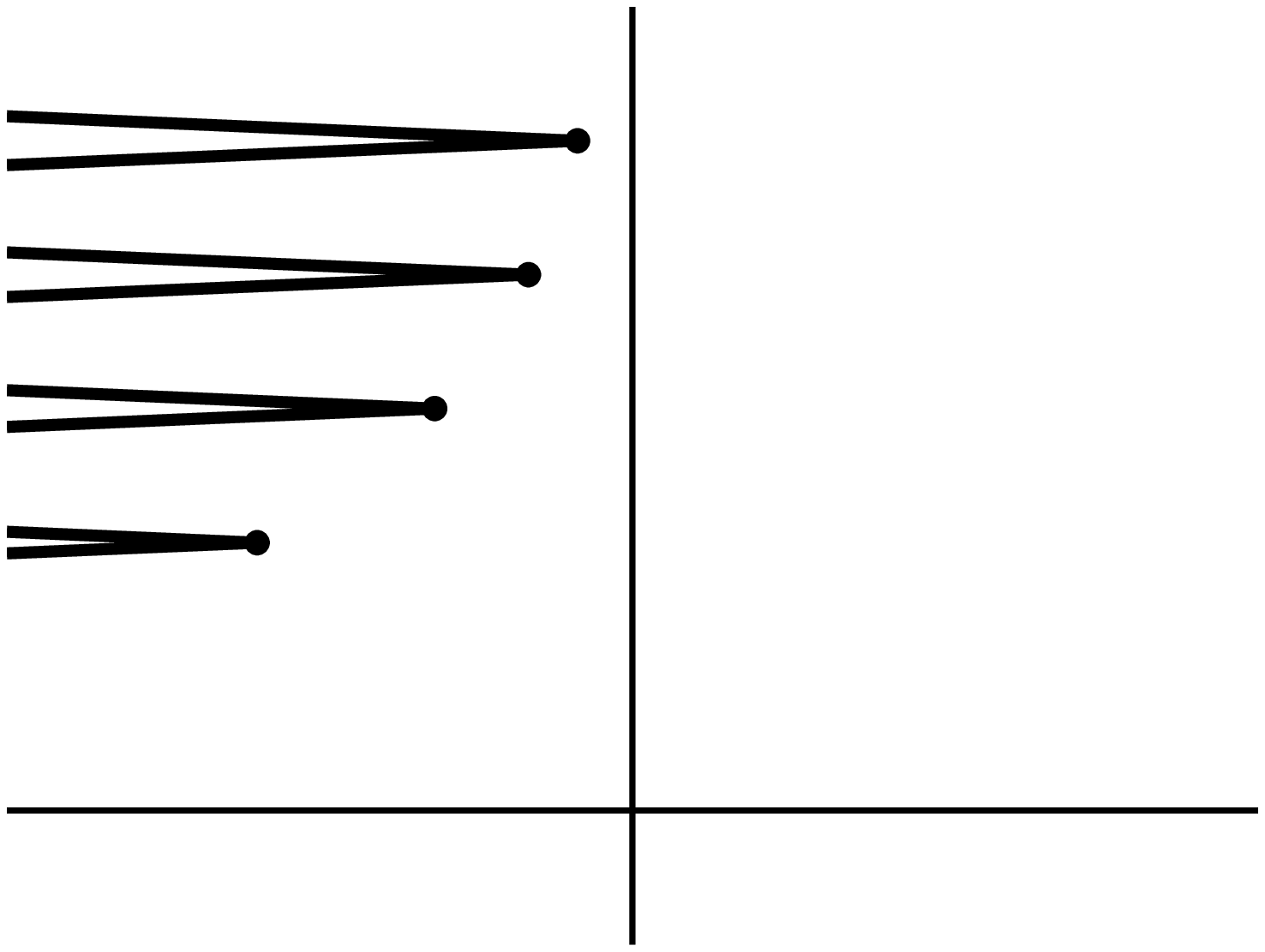}}
\put(1.5,0){\includegraphics[height=3cm]{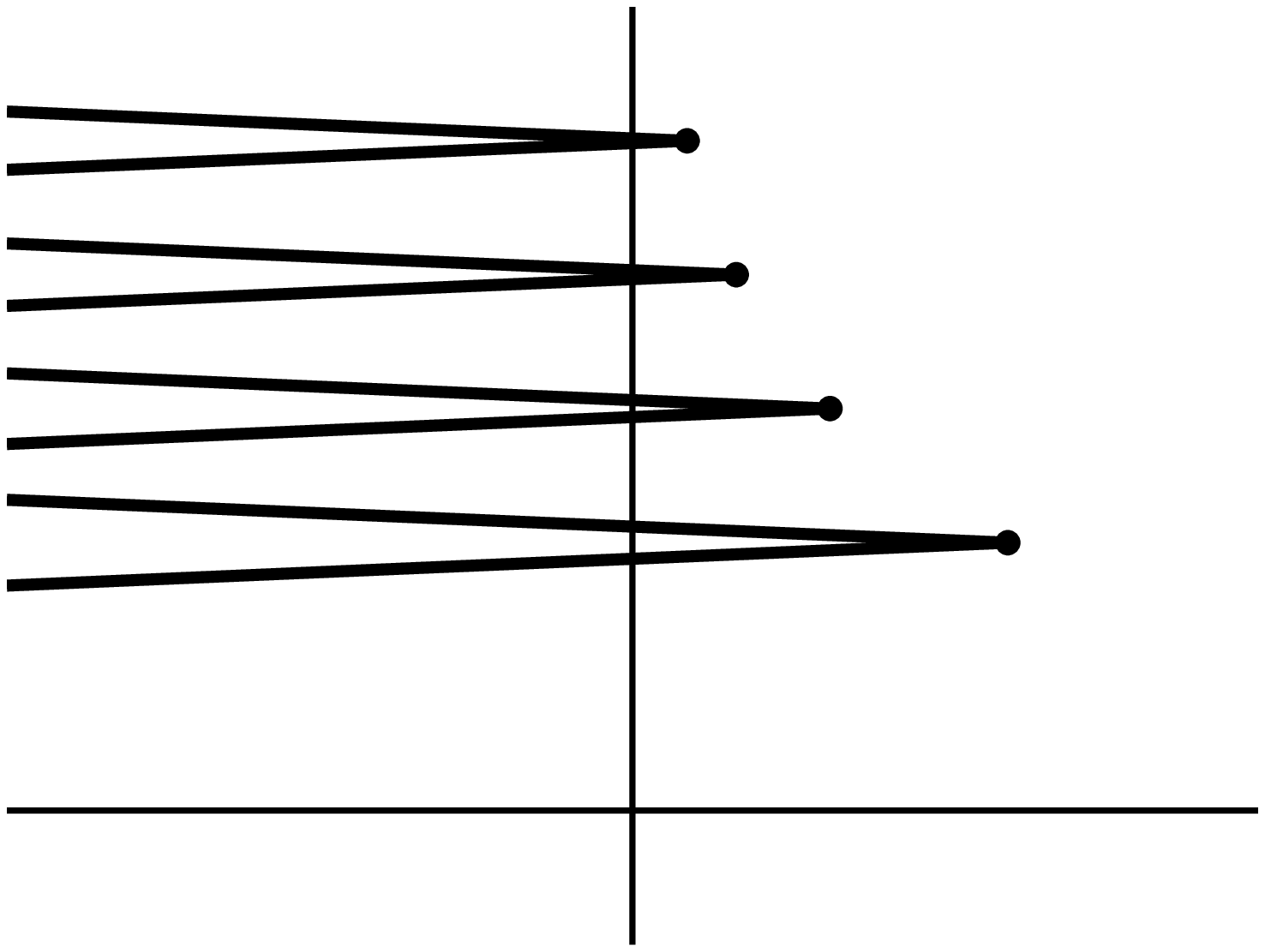}}
\put(7,0){\includegraphics[height=3cm]{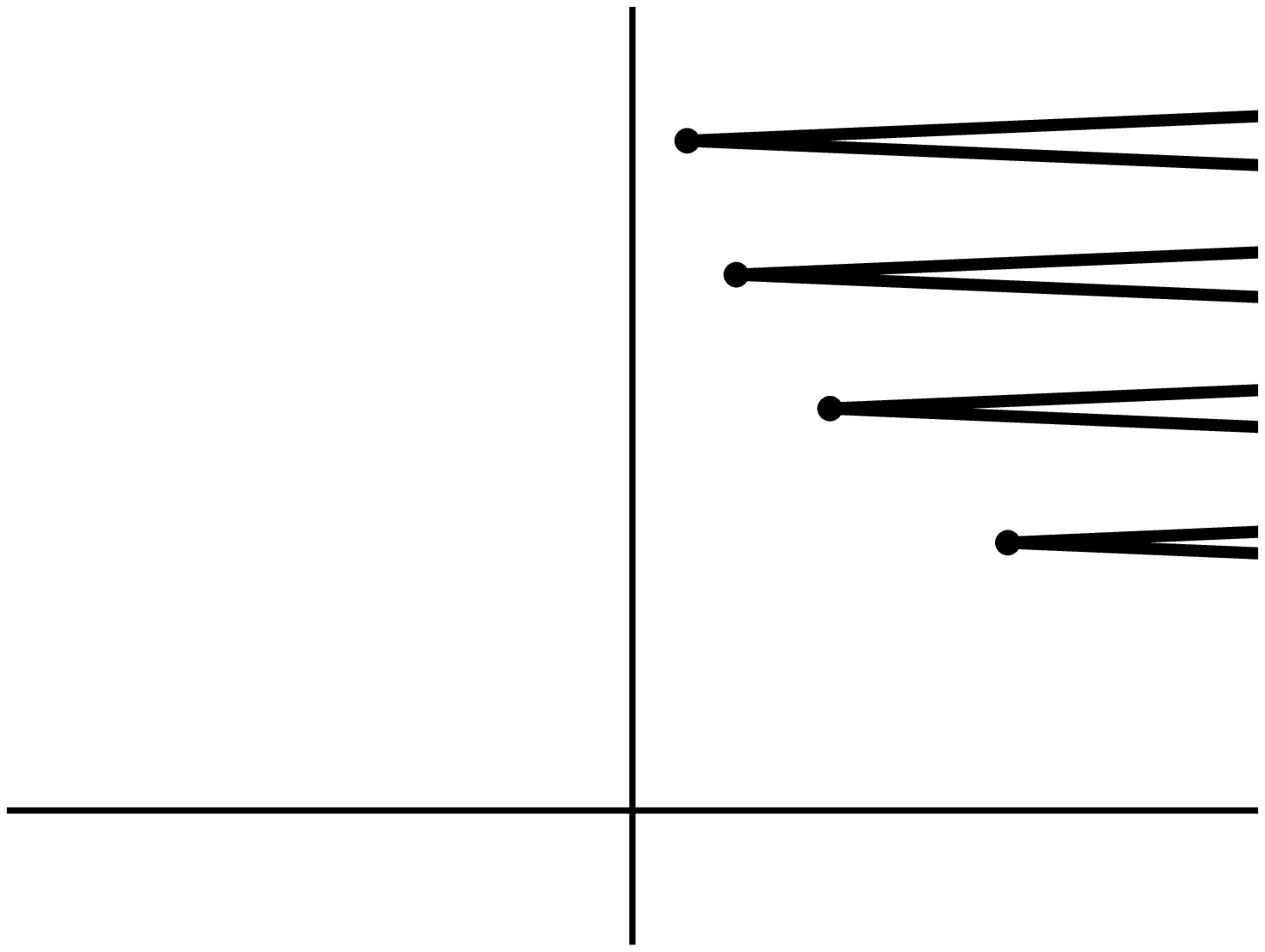}}
\put(0,4.95){\small $c_2 < 0$}
\put(0,1.75){\small $c_2 > 0$}
\put(2.8,6.5){\small $b^{\cX}_2 < 0$}
\put(8.3,6.5){\small $b^{\cX}_2 > 0$}
\put(5.23,3.77){\scriptsize $\xi$}
\put(3.56,6.04){\scriptsize $n+x$}
\put(10.73,3.77){\scriptsize $\xi$}
\put(9.06,6.04){\scriptsize $n+x$}
\put(5.23,.57){\scriptsize $\xi$}
\put(3.56,2.84){\scriptsize $n+x$}
\put(10.73,.57){\scriptsize $\xi$}
\put(9.06,2.84){\scriptsize $n+x$}
\end{picture}
\caption{
Schematic bifurcation diagrams indicating single-round periodic solutions and their BCBs,
as governed by the signs of $c_2$ and $b^{\cX}_2$.
\label{fig:bifDiagsHCCorner}
}
\end{center}
\end{figure}

We can use Theorem \ref{th:bcbs} to sketch bifurcation diagrams.
By (\ref{eq:xik}) the sign of $c_2$ dictates the side of $\xi = 0$ on which the BCBs occur,
whilst by (\ref{eq:xyS1kadmissible}) the sign of $b^{\cX}_2 c_2$ dictates the side of each $\xi_k$
on which the periodic solutions exist.
This is summarised in Fig.~\ref{fig:bifDiagsHCCorner}.

\begin{proof} 
The fixed points $\left( x^{\cX}_k, y^{\cX}_k \right)$ and $\left( x^{\cY}_k, y^{\cY}_k \right)$
coincide on $f^{r-s}(\Sigma)$ when $y^{\cS}_k(\xi) = g \left( x^{\cS}_k(\xi); \xi \right)$.
By combining the expression for the fixed points (\ref{eq:xyS1k}) with the expression for $g$ (\ref{eq:g}), we obtain
\begin{equation}
y^{\cS}_k(\xi) - g \left( x^{\cS}_k(\xi); \xi \right) =
\frac{\sigma^{-k} - c_2 \xi}{b^{\cS}_2} \sigma^{-k} +
\cObig \left( \lambda^k \right) + \cObig \left( \sigma^{-2k} \right) \;,
\label{eq:yhat}
\end{equation}
which is zero for a unique value of $\xi = \cObig \left( \sigma^{-k} \right)$ satisfying (\ref{eq:xik}).

The fixed point $\left( x^{\cX}_k(\xi), y^{\cX}_k(\xi) \right)$
is admissible if $y^{\cX}_k(\xi) \le g \left( x^{\cX}_k(\xi); \xi \right)$
and the fixed point $\left( x^{\cY}_k(\xi), y^{\cY}_k(\xi) \right)$
is admissible if $y^{\cY}_k(\xi) \ge g \left( x^{\cY}_k(\xi); \xi \right)$.
By (\ref{eq:yhat}),
\begin{equation}
{\rm sgn} \left( y^{\cS}_k(\xi) - g \left( x^{\cS}_k(\xi); \xi \right) \right) =
- {\rm sgn} \left( b^{\cS}_2 c_2 (\xi - \xi_k) \right) \;,
\end{equation}
which implies (\ref{eq:xyS1kadmissible}).

Lastly, the eigenvalues of the periodic solutions are are the roots of the characteristic polynomial
$\gamma^2 - \tau^{\cS}_k \gamma + \delta^{\cS}_k$.
The above expressions for $\tau^{\cS}_k$ and $\delta^{\cS}_k$,
(\ref{eq:tauS1k}) and (\ref{eq:deltaS1k}), directly lead to (\ref{eq:gammaus}).
\end{proof}

\subsubsection*{A coordinate change to the border-collision normal form}

Here we introduce $k$-dependent coordinates centred about the BCB:
\begin{equation}
\hat{x} = x - x^{\cX}_k(\xi_k) \;, \qquad
\hat{y} = y - g(x;\xi) \;, \qquad
\hat{\xi} = \xi - \xi_k \;.
\label{eq:hattedCoords}
\end{equation}
Note that $x^{\cX}_k(\xi_k) = x^{\cY}_k(\xi_k)$
(because the two periodic solutions coincide at the BCB),
hence our choice of $\cX$ in (\ref{eq:hattedCoords}) does not generate an asymmetry.
In these coordinates the BCB occurs at $(\hat{x},\hat{y};\hat{\xi}) = (0,0;0)$.
The map $f^{r+k}$, as given by (\ref{eq:frpk}), becomes 
\begin{equation}
\begin{bmatrix} \hat{x} \\ \hat{y} \end{bmatrix} \mapsto
\begin{cases}
\begin{bmatrix} \hat{a}_1 & \hat{b}^{\cX}_1 \\ \hat{a}_2 & \hat{b}^{\cX}_2 \end{bmatrix}
\begin{bmatrix} \hat{x} \\ \hat{y} \end{bmatrix} +
\begin{bmatrix} \hat{c}_1 \\ \hat{c}_2 \end{bmatrix} \hat{\xi} +
\cObig \left( \| \hat{x}, \hat{y}, \hat{\xi} \|^2 \right) \;, & \hat{y} \le 0 \\
\begin{bmatrix} \hat{a}_1 & \hat{b}^{\cY}_1 \\ \hat{a}_2 & \hat{b}^{\cY}_2 \end{bmatrix}
\begin{bmatrix} \hat{x} \\ \hat{y} \end{bmatrix} +
\begin{bmatrix} \hat{c}_1 \\ \hat{c}_2 \end{bmatrix} \hat{\xi} +
\cObig \left( \| \hat{x}, \hat{y}, \hat{\xi} \|^2 \right) \;, & \hat{y} \ge 0
\end{cases} \;,
\label{eq:hattedfrpk}
\end{equation}
where
\begin{align}
\hat{a}_1 &= a_1 \lambda^k + \cObig \left( \lambda^k \sigma^{-k} \right) \;, &
\hat{b}^{\cS}_1 &= b^{\cS}_1 \lambda^k + \cObig \left( \lambda^k \sigma^{-k} \right) \;, &
\hat{c}_1 &= c_1 \lambda^k + \cObig \left( \lambda^k \sigma^{-k} \right) \;, \label{eq:hattedParams1} \\
\hat{a}_2 &= a_2 \sigma^k + \cO(1) \;, &
\hat{b}^{\cS}_2 &= b^{\cS}_2 \sigma^k + \cO(1) \;, &
\hat{c}_2 &= c_2 \sigma^k + \cO(1) \;, \label{eq:hattedParams2}
\end{align}
for $\cS = \cX,\cY$.

Next we apply a second coordinate change to put the map into the border-collision normal form (\ref{eq:bcNormalForm}).
With 
\begin{equation}
\tilde{x} = \hat{y} \;, \qquad
\tilde{y} = \hat{a}_2 \hat{x} + \hat{a}_1 \hat{y} +
\left( \hat{a}_1 \hat{c}_2 - \hat{a}_2 \hat{c}_1 \right) \hat{\xi} \;, \qquad
\tilde{\xi} = \hat{c}_2 \hat{\xi} \;,
\label{eq:tildedCoords}
\end{equation}
the map (\ref{eq:hattedfrpk}) becomes
\begin{equation}
\begin{bmatrix} \tilde{x} \\ \tilde{y} \end{bmatrix} \mapsto \begin{cases}
\begin{bmatrix} \tau^{\cX}_k & 1 \\ -\delta^{\cX}_k & 0 \end{bmatrix}
\begin{bmatrix} \tilde{x} \\ \tilde{y} \end{bmatrix} +
\begin{bmatrix} \tilde{\xi} \\ 0 \end{bmatrix} \;, & \tilde{x} \le 0 \\
\begin{bmatrix} \tau^{\cY}_k & 1 \\ -\delta^{\cY}_k & 0 \end{bmatrix}
\begin{bmatrix} \tilde{x} \\ \tilde{y} \end{bmatrix} +
\begin{bmatrix} \tilde{\xi} \\ 0 \end{bmatrix} \;, & \tilde{x} \ge 0
\end{cases} \;,
\label{eq:bcNormalForm2}
\end{equation}
omitting nonlinear terms in the two pieces of the map,
and where $\tau^{\cX}_k$, $\delta^{\cX}_k$, $\tau^{\cY}_k$ and $\delta^{\cY}_k$
are given by (\ref{eq:tauS1k}) and (\ref{eq:deltaS1k}).
The coordinate change (\ref{eq:tildedCoords}) is invertible
because $\hat{a}_2$ and $\hat{c}_2$ are nonzero,
which is a consequence of (\ref{eq:a2nonzero}), (\ref{eq:c2nonzero}),
and (\ref{eq:hattedParams2}).
This also demonstrates that (\ref{eq:hattedfrpk}) is ``observable'' \cite{Si16,Di03} -- a
necessary condition for (\ref{eq:hattedfrpk})
to be transformable to the border-collision normal form.

\subsubsection*{Dynamics created in the border-collision bifurcations}

\begin{figure}[b!]
\begin{center}
\setlength{\unitlength}{1cm}
\begin{picture}(8,6)
\put(0,0){\includegraphics[height=6cm]{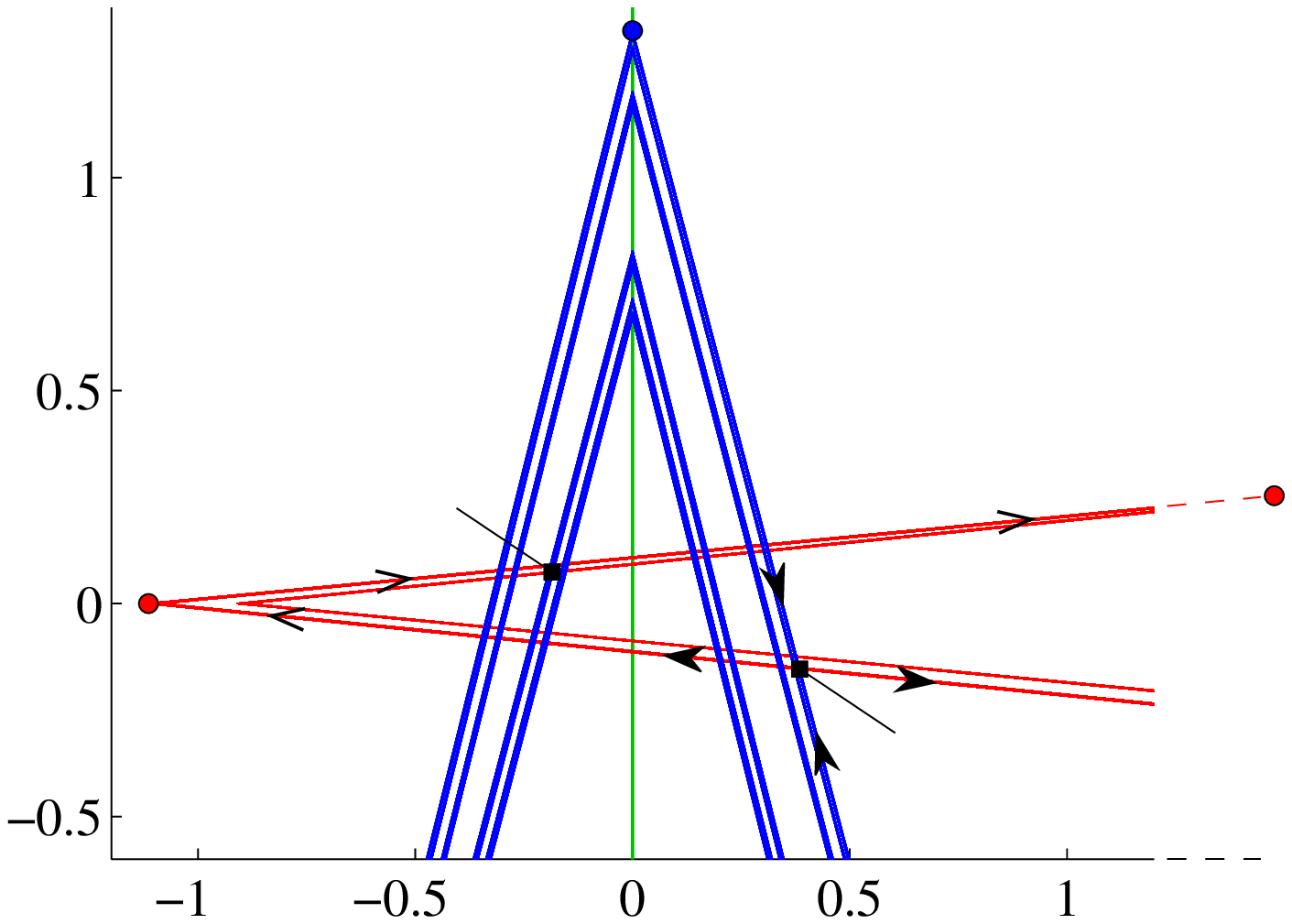}}
\put(3.85,0){\small $\tilde{x}$}
\put(0,2.27){\small $\tilde{y}$}
\put(4.12,5.7){\scriptsize $z_{-1}$}
\put(1.02,2.12){\scriptsize $z_1$}
\put(7.58,3.2){\scriptsize $z_2$}
\put(1.68,3){\scriptsize $\left( \tilde{x}^{\cX}_k, \tilde{y}^{\cX}_k \right)$}
\put(5.58,1.36){\scriptsize $\left( \tilde{x}^{\cY}_k, \tilde{y}^{\cY}_k \right)$}
\end{picture}
\caption{
A phase portrait of (\ref{eq:bcNormalForm2}) with (\ref{eq:tildedParams}) and $\tilde{\xi} = -1$,
showing the stable and unstable manifolds of $\left( \tilde{x}^{\cY}_k, \tilde{y}^{\cY}_k \right)$.
\label{fig:invMansHCCorner}
}
\end{center}
\end{figure}

We now use (\ref{eq:bcNormalForm2}) to analyse dynamics created in the BCB.
By (\ref{eq:tauS1k}), the traces $\tau^{\cX}_k$ and $\tau^{\cY}_k$ are large and of opposing signs.
By (\ref{eq:deltaS1k}), the determinants $\delta^{\cX}_k$ and $\delta^{\cY}_k$ are small.
As an example, we use
\begin{equation}
\tau^{\cX}_k = -4 \;, \qquad
\delta^{\cX}_k = 0.4 \;, \qquad
\tau^{\cY}_k = 4 \;, \qquad
\delta^{\cY}_k = 0.4 \;.
\label{eq:tildedParams}
\end{equation}
This example uses positive values for both determinants
which corresponds to the case that $f^r$ is locally invertible and orientation-preserving.
With instead larger values for the traces and smaller values for the determinants,
the dynamics of the map is similar but features are more difficult to discern
in a single phase portrait.

The structure of the dynamics of (\ref{eq:bcNormalForm2})
is independent of the magnitude of $\tilde{\xi}$.
This is due to the piecewise-linear nature of the border-collision normal form \cite{Si16}.
It therefore suffices to consider three values for $\tilde{\xi}$, say $-1$, $0$ and $1$.
The BCB occurs at $\tilde{\xi} = 0$.
With $\tilde{\xi} = 1$, the map (\ref{eq:bcNormalForm2}) with (\ref{eq:tildedParams})
has no bounded invariant sets\removableFootnote{
This can be proved by showing that forward orbits always get trapped
on the side with the positive eigenvalues,
and then diverge in the direction of the unstable eigenvector.
}.
Fig.~\ref{fig:invMansHCCorner} shows a phase portrait of (\ref{eq:bcNormalForm2}) with
(\ref{eq:tildedParams}) and $\tilde{\xi} = -1$.
There are two saddle fixed points,
$\left( \tilde{x}^{\cX}_k, \tilde{y}^{\cX}_k \right)$ and $\left( \tilde{x}^{\cY}_k, \tilde{y}^{\cY}_k \right)$,
which correspond to the fixed points
$\left( x^{\cX}_k, y^{\cX}_k \right)$ and $\left( x^{\cY}_k, y^{\cY}_k \right)$
of (\ref{eq:xyS1k}), respectively.
The fixed point $\left( \tilde{x}^{\cY}_k, \tilde{y}^{\cY}_k \right)$ has positive eigenvalues,
and its stable and unstable manifolds have transverse intersections.
The stable and unstable manifolds of
$\left( \tilde{x}^{\cY}_k, \tilde{y}^{\cY}_k \right)$
(which has negative eigenvalues)
are almost indistinguishable from those of $\left( \tilde{x}^{\cY}_k, \tilde{y}^{\cY}_k \right)$
on the scale Fig.~\ref{fig:invMansHCCorner},
and also intersect transversally.

\begin{theorem}
Consider (\ref{eq:bcNormalForm2}) with (\ref{eq:tauS1k}) and (\ref{eq:deltaS1k})
and ${\rm sgn}(\tilde{\xi}) = {\rm sgn} \left( b^{\cX}_2 \right)$
(so that both fixed points are admissible).
For sufficiently large values of $k \in \mathbb{Z}^+$,
the stable manifolds of $\left( \tilde{x}^{\cX}_k, \tilde{y}^{\cX}_k \right)$
and $\left( \tilde{x}^{\cY}_k, \tilde{y}^{\cY}_k \right)$ transversally intersect
the unstable manifolds of these fixed points.
\label{th:transverseIntersections}
\end{theorem}

\begin{proof} 

The proof is achieved via direct calculations.
The fixed points of (\ref{eq:bcNormalForm2}) with (\ref{eq:tauS1k}) and (\ref{eq:deltaS1k})
are given by
\begin{equation}
\tilde{x}^{\cS}_k(\tilde{\xi}) = \left( \frac{-1}{b^{\cS}_2} \sigma^{-k} +
\cObig \left( \sigma^{-2k} \right) \right) \tilde{\xi} \;, \qquad
\tilde{y}^{\cS}_k(\tilde{\xi}) = \left( \frac{a_1 b^{\cS}_2 - a_2 b^{\cS}_1}{b^{\cS}_2} \lambda^k +
\cObig \left( \lambda^k \sigma^{-k} \right) \right) \tilde{\xi} \;.
\label{eq:tildexyS1k}
\end{equation}
Their eigenvalues are given by (\ref{eq:gammaus})
and the corresponding eigenvectors are
\begin{equation}
v^{\cS}_{k,u} = \begin{bmatrix} 1 \\
-\frac{a_1 b^{\cS}_2 - a_2 b^{\cS}_1}{b^{\cS}_2} \lambda^k +
\cObig \left( \lambda^k \sigma^{-k} \right) \end{bmatrix} \;, \qquad
v^{\cS}_{k,s} = \begin{bmatrix}
-\frac{1}{b^{\cS}_2} \sigma^{-k} + \cObig \left( \sigma^{-2k} \right) \\ 1 \end{bmatrix} \;.
\label{eq:vus}
\end{equation}
Suppose $b^{\cX}_2 < 0$ (the case $b^{\cX}_2 > 0$ can be proved similarly).
Then $\tilde{\xi} < 0$ and it suffices to consider $\tilde{\xi} = -1$.

Next we show that the stable and unstable manifolds
of $\left( \tilde{x}^{\cY}_k, \tilde{y}^{\cY}_k \right)$ intersect.
Intersections involving the invariant manifolds of $\left( \tilde{x}^{\cX}_k, \tilde{y}^{\cX}_k \right)$
can be proved in the same fashion and for brevity we omit such a proof.

As the stable manifold of the fixed point $\left( \tilde{x}^{\cY}_k, \tilde{y}^{\cY}_k \right)$
emanates from this point,
it is initially linear and in the direction of the stable eigenvector $v^{\cS}_{k,s}$.
The first corner (or kink) of the stable manifold is on the switching manifold ($\tilde{x} = 0$),
and we call this point $z_{-1}$, see Fig.~\ref{fig:invMansHCCorner}.
From (\ref{eq:tildexyS1k}) and (\ref{eq:vus})
we determine the $\tilde{y}$-value of $z_{-1}$ to be $1 + \cObig \left( \sigma^{-k} \right)$.

Similarly, from (\ref{eq:tildexyS1k}) and (\ref{eq:vus}) we find that the first corner
of the unstable manifold of $\left( \tilde{x}^{\cY}_k, \tilde{y}^{\cY}_k \right)$, call it $z_1$,
occurs on $\tilde{y} = 0$ (the image of the switching manifold)
with an $\tilde{x}$-value of $-1 + \cObig \left( \sigma^{-k} \right)$.
The next corner of the unstable manifold occurs at the image of $z_1$, call it $z_2$.
From (\ref{eq:tauS1k}), (\ref{eq:deltaS1k}) and (\ref{eq:bcNormalForm2}) we obtain
\begin{equation}
z_2 = \begin{bmatrix}
-b^{\cX}_2 \sigma^k + \cO(1) \\
\left( a_1 b^{\cX}_2 - a_2 b^{\cX}_1 \right) \lambda^k \sigma^k + \cObig \left( \lambda^k \right)
\end{bmatrix} \;.
\end{equation}

In summary, the stable manifold of $\left( \tilde{x}^{\cY}_k, \tilde{y}^{\cY}_k \right)$
includes the line segment from $\left( \tilde{x}^{\cY}_k, \tilde{y}^{\cY}_k \right)$ to $z_{-1}$,
and the unstable manifold of $\left( \tilde{x}^{\cY}_k, \tilde{y}^{\cY}_k \right)$
includes the line segment from $z_1$ to $z_2$.
From our expressions for these points
it is evident that these line segments intersect transversally for sufficiently large values of $k$.
\end{proof}

\section{Discussion}
\label{sec:conc}
\setcounter{equation}{0}

\subsubsection*{Summary}

This paper is the first to provide the generic unfolding of a homoclinic corner.
Theorem \ref{th:unstable} states that with parameter values
sufficiently close to the homoclinic corner (within $\ee$)
there cannot exist a stable periodic solution involving many (at least $k$) consecutive points
in a neighbourhood of the saddle each time the solution enters this neighbourhood.
Typical nearby forward orbits therefore either approach an aperiodic attractor, or diverge.
The practical effect of a homoclinic corner may be that of a crisis, as in Fig.~\ref{fig:pptRmp6},
where a chaotic attractor is destroyed.

A sequence of BCBs,
$\xi_n$, at which pairs of single-round periodic solutions are created,
limits to a homoclinic corner, see Theorem \ref{th:bcbs}.
This is analogous to the sequence of saddle-node bifurcations near a homoclinic tangency of a smooth map,
and the two sequences satisfy the same scaling law.
The dynamics local to each BCB is described by a piecewise-linear continuous map
which we transformed into the border-collision normal form (\ref{eq:bcNormalForm2}).
From this we showed that the stable and unstable manifolds of the
single-round periodic solutions have transverse intersections, Theorem \ref{th:transverseIntersections}.
At each BCB an invariant Cantor set is created on which the dynamics is chaotic.
In contrast, the dynamics local to each saddle-node bifurcation
near a homoclinic tangency of a smooth map is well approximated by a
two-dimensional quadratic map \cite{PaTa93,HoWh84},
which can be transformed into the H\'{e}non map \cite{He76,Mi87}.

The calculations given above assume $\lambda \sigma < 1$.
With instead $\lambda \sigma > 1$,
the roles of $\lambda$ and $\frac{1}{\sigma}$ are reversed.
Pairs of single-round periodic solutions are created in BCBs with the alternate scaling law,
$\xi_n = C \lambda^n + \cObig \left( \sigma^{-n} \right) + \cObig \left( \lambda^{2 n} \right)$.
But again both periodic solutions are saddles
and the dynamics near the BCBs is given by (\ref{eq:bcNormalForm2}).

\subsubsection*{Reduction to a skew tent map}

Since the determinants $\delta^{\cX}_k$ and $\delta^{\cY}_k$ in (\ref{eq:bcNormalForm2}) are small,
the $\tilde{x}$-values of iterates of (\ref{eq:bcNormalForm2}) are well approximated by the one-dimensional map
\begin{equation}
\tilde{x} \mapsto \begin{cases}
\tau^{\cX}_k \tilde{x} + \tilde{\xi} \;, & \tilde{x} \le 0 \\
\tau^{\cY}_k \tilde{x} + \tilde{\xi} \;, & \tilde{x} \ge 0
\end{cases} \;.
\label{eq:skewTentMap}
\end{equation}
This is a steep skew tent map.
The analogous one-dimensional map for a homoclinic tangency is quadratic.
Each BCB can therefore be viewed as the contraction into a single point
of the well known bifurcation sequence of quadratic maps
(a saddle-node bifurcation, a period-doubling cascade to chaos,
windows of periodicity and bifurcations associated with the chaotic regime,
and finally the loss of a local attractor in a homoclinic bifurcation).
Arbitrarily steep quadratic maps can be renormalised,
but the piecewise-linear nature of (\ref{eq:skewTentMap})
inhibits this manner of renormalisation.

\subsubsection*{Homoclinic corners in higher-dimensional maps}

For a homoclinic tangency to a saddle fixed point in a high-dimensional smooth map,
only the slowest stable and unstable directions related to the fixed point are important
for understanding the dynamics close to the tangency \cite{GoSh96,GoSh97}.
This is because typical forward and backward orbits approach the fixed point tangent to the slow manifolds.
If the eigenvalues corresponding to the slowest stable and unstable directions
are both real, generically the limiting dynamics is governed by a one-dimensional quadratic map,
as in the two-dimensional case.
Otherwise one or both slow directions correspond to complex eigenvalues,
and there are a total of four different limiting maps that arise in generic situations \cite{GoTu05}.

For homoclinic corners in higher dimensions
we expect there to be a similar set of limiting maps.
As with (\ref{eq:skewTentMap}), we expect that these maps are piecewise-linear
but involve arbitrarily large coefficients
and for this reason lack attracting invariant sets.

\subsubsection*{Outlook for future studies}

In addition to accommodating higher dimensions, several problems remain for future work.
The codimension-two homoclinic corner of Fig.~\ref{fig:tonguesExact}
appears to be a limit point of mode-locking regions.
This phenomenon is evident in a similar codimension-two scenario for a PWS discontinuous map \cite{Mi13}
and could be explained by a generalisation of the unfolding performed here.
It would also be helpful to study homoclinic corners as the limit of homoclinic tangencies
by smoothing the map (\ref{eq:f})
so that we can understand how the stability of periodic solutions is lost in the nonsmooth limit.

\end{document}